\documentclass[12pt]{amsart}
\usepackage[T1]{fontenc}
\usepackage[utf8]{inputenc}
\usepackage{fullpage}
\usepackage{algorithm}%
\usepackage{algpseudocode}%
\usepackage{amssymb,amsthm,amsrefs}
\usepackage{dsfont} 
\usepackage{enumitem}
\usepackage{graphicx}
\usepackage{hyperref}
\usepackage{mathtools}
\usepackage{microtype} 
\usepackage{stmaryrd}
\usepackage{xcolor}
\newcommand{\Norm}[2]{\|#1\|\left.\vphantom{T_{j_0}^0}\!\!\right._{#2}}

\newcommand{\R}{{\mathbb R}}

\newcommand{\Vc}{\mathcal{V}}

\newcommand{\Pc}{\mathcal{P}}

\newcommand{\Hc}{\mathcal{H}}

\newcommand{\Kc}{{\mathcal{K}}}

\newtheorem{theorem}{Theorem}
\newtheorem{proposition}{Proposition}
\newtheorem{lemma}{Lemma}
\newtheorem{corollary}{Corollary}
\newtheorem{remark}{Remark}
\newtheorem{assumption}{Assumption}

\newcommand{\Span}{\mathrm{Span}}

\DeclareMathOperator*{\argmax}{arg\,max}

\title{A generalized spectral concentration problem and \\ the varying masks algorithm}
\author{E. Faou, Y. Le H\'enaff}
\date{}

\begin{document}

\begin{abstract}
    In this paper we generalize the spectral concentration problem as formulated by Slepian, Pollak and Landau in the 1960s.
    We show that a generalized version with arbitrary space and Fourier masks is well-posed, and we prove some new results concerning general quadratic domains and gaussian filters. We also propose a more general splitting representation of the spectral concentration operator allowing to construct quasi-modes in some situations. We then study its discretization and we illustrate the fact that standard eigen-algorithms are not robust because of a clustering of eigenvalues. We propose a new alternative algorithm that can be implemented in any dimension and for any domain shape, and that gives very efficient results in practice. 
\end{abstract}

\maketitle

\section{Introduction}
\label{sect:slepian -- introduction}

The spectral concentration problem was posed by Slepian, Landau and Pollak in 1961 \cite{slepianProlateSpheroidalWave1961}. 
This problems stems from the well known fact that a function $f \in L^2$ with fixed $L^2$ norm cannot be both concentrated in space {\em and} in Fourier, owing to the celebrated Heisenberg inequality 
$$
\Norm{xf}{L^2}^2 \Norm{\xi\widehat f}{L^2}^2 \geq \frac{\Norm{f}{L^2}^4}{16 \pi^2}
$$
where $\widehat f$ is the Fourier transform of $f$. 
The spectral concentration problem studied by Slepian, Landau and Pollak was to find the function maximizing the \( L^2([-1, 1]) \) norm of a function for which its Fourier transform is supported in \( [-c, c] \) for some given parameter \( c >0 \): 
$$
\argmax \left\{ \frac{\int_{-1}^1 |f|^2}{\int_{-\infty}^\infty|f|^2 } \; \right.\left| \; \mathrm{supp}{ \widehat f} \subset [-c,c]\right\}. 
$$
It can be shown to be equivalent after some rescaling to the explicit eigenvalue problem of finding eigenpairs $(\lambda,\psi)$ of the problem  
$$
\lambda \psi (x)  = \int_{-1}^1 \frac{\sin( c(x - y))}{x-y} \psi(y) d y =: (\Kc \psi)(x), \quad x \in [-1,1]. 
$$  
In a series of papers \cites{slepianProlateSpheroidalWave1961,landauProlateSpheroidalWave1961,landauProlateSpheroidalWave1962,slepianProlateSpheroidalWave1964,slepianProlateSpheroidalWave1978}, they gave a very satisfying and efficient answer to the above setting by finding an explicit second order operator $\mathcal{P}$ commuting with $\Kc$: 
\begin{equation}
\label{Pc1}
\Pc \Kc = \mathcal{K} \Pc, \qquad \Pc = - \partial_x ( 1 - x^2) \partial_x + c^2 x^2. 
\end{equation}
The eigenvalues of $\mathcal{P}$ are special functions known as the prolate spheroidal wave functions or simply Slepian functions, giving automatically eigenfunctions for the operator $\Kc$. The analysis of the eigenvalue distribution $\lambda$ has also received a lot of attention, see \cites{LandauEigenvalueDistributionTime1980,karnikImprovedBoundsEigenvalues2021} and the references therein. Note that the previous one-dimensional analysis obviously extends to higher dimension by tensorization, restricting however the analysis to cartesian products of intervals. 

A few years later, Brander and DeFacio \cite{branderGeneralisationSlepianSolution1986} showed that when the similar problem of space and Fourier Gaussian filtering is considered, it is also possible to find a commuting differential operator (the quantum harmonic oscillator operator, correctly rescaled) and eigenvectors made of scaled Hermite functions. 

Apart from the interval and the Gaussian cases, very few results are known about the existence of a second-order commuting differential operator in other settings. The work of Gr\"unbaum, Longhi and Perlstadt \cite{grunbaumDifferentialOperatorsCommuting1982} even seems to point towards the fact that it does not exist in general except exceptional situations.

Since the spectral concentration problem has been formulated in 1961, it has found numerous applications in different fields of physics (see e.g. the review by Wang \cite{wangReviewProlateSpheroidal2017}).
Karnik \textit{et al.} even proposed a \emph{Fast Slepian transform} \cite{karnikFastSlepianTransform2019}, underlining the importance of this problem in applications.

For general domain in higher dimension, the description of the spectrum and eigenvectors remains largely unknown. In \cite{simonsSpatiospectralConcentrationCartesian2011}, Simons and Wang considered a spectral concentration problem that cannot be reduced to the historical problem on the interval or to the gaussian filter problem,  and they had to resort to purely numerical solutions.
Their numerical experiments seemed to indicate that solving the spectral concentration problem with arbitrary space and Fourier restrictions is not an issue, but this is essentially due to the favorable numerical parameters. In practice, 
simple experiments show that solving numerically a spectral concentration problem can be a source of dramatic instabilities. 
This is a purely numerical issue, already present in the interval case, and it boils down to finding eigenvectors of a matrix for which the eigenvalues are very close to each other, forming almost {\em large clusters} of eigenvalues preventing standard algorithms to catch physically relevant eigenvectors. Typically for domains, the most relevant eigenvectors are associated with eigenvalues extremely close to $1$ and it is numerically extremely difficult to classify them at a reasonable cost. The usual method used to circumvent these instabilities is to use a very large number of discretization points, at a cost that becomes rapidly prohibitive in dimensions 2 or 3 -- which are the most widely used cases for applications. 

The main goals of this paper are the following: 
\begin{itemize}
\item We define and study a generalized spectral concentration problem in any dimension covering the previous situations and we give some basic properties of the associated spectrum. 
\item We give some examples where the spectrum can be calculated and estimated, in particular for general {\em quadratic} domains. We also use a general representation of the spectral concentration operator as a {\em Strang splitting} operator allowing to have an exact expression of eigenvalues and eigenvectors for Gaussians filters. We also give a method for constructing {\em quasi-modes} using commutators of the Baker-Campbell-Hausdorff formula for general filters close to the identity. 
\item Eventually, we propose a new algorithm for approximating the eigenpairs of the spectral concentration operator. This methods that we call the {\em varying mask} algorithm allows to 
track relevant eigenvectors by letting the size of the domain vary. We report in one and two-dimensional examples the excellent behavior of this method compared to standard eigendecomposition algorithms. 
\end{itemize}

{\bf Acknowledgment}. The authors would like to thank Pierre Vernaz-Gris for several stimulating discussions of this subject. This work was conducted within the the France 2030 program, Centre Henri Lebesgue ANR-11-LABX-0020-01.

\subsection{Notations}

The Fourier transform on \( L^2(\mathbb{R}^d) \) and its inverse are chosen respectively as follows:
\begin{equation}
    \label{eqn: anum -- definition fourier transform}
    \mathcal{F}[f](\xi) := \int_{\mathbb{R}^d} f(x) e^{-i\xi\cdot x} dx,
\quad \mbox{and} \quad    \mathcal{F}^{-1}[f](x) 
    := \frac{1}{(2\pi)^d} \int_{\mathbb{R}^d} f(\xi) e^{ix \cdot\xi} d\xi.
\end{equation}
We may use the shorthand \( \widehat{f} \) to denote \( \mathcal{F}[f] \). 
Some useful classical properties of the Fourier transform are the following: for \( f,g\in L^2(\mathbb{R}^d) \),
\begin{itemize}
    \item \( (2\pi)^d \int_{\mathbb{R}^d} f(x) \overline{g(x)} dx = \int_{\mathbb{R}^d} \mathcal{F}[f](\xi) \overline{\mathcal{F}[g](\xi)} d\xi \),
    \item \( \mathcal{F}[f(\cdot - a)](\xi) = e^{-i\xi\cdot a} \mathcal{F}[f](\xi) \), for \( \xi, a \in \mathbb{R}^d \),
    \item \( \mathcal{F}(f \ast g) = \mathcal{F}(f) \mathcal{F}(g) \),
\end{itemize}
where \( \ast \) denotes the convolution operator: \( (f\ast g)(x) = \int_{\mathbb{R}^d} f(y) g(x-y) dy \).

The \( L^2 \) inner product will be denoted by \( (\cdot, \cdot)_{L^2} \) and it is taken anti-hermitian in its second argument. $\Norm{f}{L^p}$ denote the standard $L^p$ norm of a given function $f$. 

We write \( \mathbf{A}\in \mathbb{K}^{m\times n} \) to denote a matrix with components in \( \mathbb{K} \) with \( m \) rows and \( n \) columns.
For a given function \( f\in L^2(\mathbb{R}^d) \) and a given matrix \( \mathbf{A}\in \mathbb{R}^{d\times d} \), we use the shorthand \( f\circ \mathbf{A}(x) := f\left( \mathbf{A}x \right) \).

A ball centered at \( c\in \mathbb{R}^d \) with radius \( r>0 \) is denoted \( B(c, r) \). For any \( p\in \mathbb{R}^d \), \( \tau_p \) denotes a translation by \( p \): \( \tau_p(x) := x-p \).

For two integers \( m < n \in \mathbb{Z} \), we write \( [\![m, n ]\!] = [m, n] \cap \mathbb{Z} \).

Every index \( {}_S \) denotes a quantity that is related to the space domain, and every index \( {}_F \) denotes a quantity that is related to the Fourier domain.

\section{Derivation of the generalized spectral concentration problem}
\label{sect:slepian -- derivation of the generalized concentration problem}

Let \( m_S, \widehat{m_F} \in L^2(\mathbb{R}^d) \). 
The function \( m_S \) will be called the \emph{space filter}, or \emph{space mask}, and \( \widehat{m_F} \) will be called the \emph{Fourier filter}, or \emph{Fourier mask}.
We define the following operators: for \( f\in L^2(\mathbb{R}^d) \),
\begin{equation*}
    (\mathcal{M}_S f)(x) := m_S(x) f(x), 
    \quad \text{ and }\quad
    (\mathcal{M}_F f)(x) := \mathcal{F}^{-1}\left[\widehat{m_F} \mathcal{F}[f]\right](x).
\end{equation*}

In this work we will consider a composition of these operators, more specifically \( \mathcal{M}_F \mathcal{M}_S \).
Most of what follows would also apply if we considered \( \mathcal{M}_S \mathcal{M}_F \), with small modifications.

One has 
\begin{equation*}
    (\mathcal{M}_F \mathcal{M}_S f)(x)
    = \mathcal{F}^{-1}\left[ \widehat{m_F} \mathcal{F}[m_S f] \right](x)
    = (m_F \ast (m_S f))(x).
\end{equation*}
Owing to the fact that the Fourier transform of the Dirac distribution \( \delta_0 \) is the identity function, i.e. \( \int_{\mathbb{R}^d} e^{i\eta\cdot x} dx = (2\pi)^d \delta_0(\eta) \), direct computations yield
\begin{align}
     \Norm{\mathcal{M}_F \mathcal{M}_S f}{L^2}^2 \nonumber &= \frac{1}{(2\pi)^d}  \int_{\mathbb{R}^d} \int_{\mathbb{R}^d} m_S(x) f(x) \overline{m_S(y) f(y)} \left( \int_{\mathbb{R}^d} e^{i\xi\cdot (x-y)} \left|\widehat{m_F}(\xi)\right|^2 d\xi \right) dydx \nonumber  \\
    & = \left( \mathcal{K} f, f \right)_{L^2}.
   \label{eqn: slepian -- norm2 MF MS}
\end{align}
The operator \( \mathcal{K} \) is the \emph{generalized concentration operator} on \( L^2(\mathbb{R}^d) \) that we have defined by the formula
\begin{equation}
    \label{K2}
    \Kc = \mathcal{M}_S^* \mathcal{M}_F^* \mathcal{M}_F\mathcal{M}_S,
\end{equation}
where $\mathcal{M}_S^*$ and $\mathcal{M}_F^*$ are the $L^2$-adjoint of the operators $\mathcal{M}_S$ and $\mathcal{M}_F$ respectively. We thus have 
\begin{equation}
    \label{defK}
    \left|
    \begin{array}{l}
    \displaystyle
        (\mathcal{K}f)(x) := \overline{m_S(x)} \int_{\mathbb{R}^d} k(x - y) m_S(y) f(y)  dy.\\
        \displaystyle k(z) = \frac{1}{(2\pi)^d} \int_{\mathbb{R}^d} e^{i\xi\cdot z} \left|\widehat{m_F}(\xi)\right|^2 d\xi .
    \end{array}
    \right.
\end{equation}

\begin{lemma}
    \label{lemma:slepian -- generalised kernel is L2}
    If \( m_S, \widehat{m_F} \in L^2(\mathbb{R}^d) \), then for all $f \in L^2(\R^d)$,     \begin{equation}
       \Norm{\Kc f}{L^2} \leq \Norm{\widehat{m_F}}{L^2}^2 \Norm{m_S}{L^2}^2 \Norm{f}{L^2}. \end{equation}
\end{lemma}

\begin{proof}
    It is a consequence of Young and H\"older inequalities which yield
    $$
    \Norm{\mathcal{M}_F \mathcal{M}_S f}{L^2} = 
    \Norm{m_F * (m_S f)}{L^2} \leq \Norm{m_F}{L^2} \Norm{m_S f}{L^1} \leq
    \Norm{\widehat{m_F}}{L^2} \Norm{m_S}{L^2} \Norm{f}{L^2}
    $$
    and the definition of $\Kc$. 
\end{proof}

For any function \( f\in L^2(\mathbb{R}^d) \), we define the associated \emph{concentration ratio}:
\begin{equation}
\label{eqn: slepian -- maximization for generalized problem}
    \nu(f) := \frac{\| \mathcal{M}_F \mathcal{M}_S f\|^2_{L^2}}{\|f\|^2_{L^2}} = \frac{(\mathcal{K}f, f)_{L^2}}{(f, f)_{L^2}}.
\end{equation}
It is clear that \( \nu \geq 0 \) and that \( \nu \leq \|m_S\|^2_{L^2} \|\widehat{m}_F\|^2_{L^2} \) using Lemma \ref{lemma:slepian -- generalised kernel is L2}.
In this work we are interested in finding the functions that maximize the concentration ratio, {\em i.e.} solutions to the eigenvalue problem 
\begin{equation}
\label{EVpb}
        \mbox{Find}\, (\lambda,\psi) \in \R \times L^2\quad \text{s.t.} \quad
    \Kc \psi = \lambda \psi. 
\end{equation}
and to compute efficiently the eigenvectors corresponding to the largest eigenvalues $\lambda$. 
This problem is equivalent to finding the singular values of the operator $\mathcal{M}_F \mathcal{M}_S$. 

\subsection{Properties}

Using classical results (see e.g. \cite{reedMethodsModernMathematical1980}*{Chapter~VI}), one gets the following properties:
\begin{proposition}
    \label{prop:slepian -- properties of the operator K}
    The concentration operator \( \mathcal{K} \) defined in \eqref{defK} enjoys the following properties:
    \begin{enumerate}[label*=\emph{\alph*)}]
        \item  \( \mathcal{K} \) is a Hilbert-Schmidt operator, self-adjoint, compact, and positive semi-definite.
        \item The countable family \( \left\{ \psi_i \right\}_{i=1}^\infty \) of eigenfunctions of \( \mathcal{K} \) is orthonormal for the usual \( L^2(\mathbb{R}^d) \) inner product and complete in \( L^2(\mathbb{R}^d) \).
        The associated eigenvalues \( \left\{ \lambda_i \right\}_{i=1}^\infty \) are real, nonnegative, and we can order them so that \( \ \lambda_i \geq \lambda_{i+1} \geq 0 \) for \( i \geq 1 \).
        \item The orthonormal basis of eigenfunctions \( \left\{ \psi_i \right\}_{i=1}^\infty \) are critical points for the concentration ratio \eqref{eqn: slepian -- maximization for generalized problem}, and can be obtained by the successive maximization problems
        \begin{equation}
            \label{minmax}\lambda_n = \sup_{\substack{f\in L^2(\mathbb{R}^d)\\f\in [\psi_1, \dots, \psi_{n-1}]^\perp}} \frac{(\mathcal{K} f, f)_{L^2}}{(f,f)_{L^2}},
        \end{equation}
        where \( [\psi_1, \dots, \psi_{n-1}]^\perp := \{u\in L^2(\mathbb{R}^d) : (u, \psi_i)_{L^2} = 0, i=1, \dots, n-1 \}\).
        \item For large \( n \), \( \lambda_n = o(n^{-1/2}) \).
        \item Suppose \( |\widehat{m_F}|^2 \) is even, and \( m_S \) is real, then \( \mathcal{K} \) is real-valued for real inputs.
    \end{enumerate}
\end{proposition}

\begin{proof}
    The fact that \( \mathcal{K} \) is a Hilbert-Schmidt operator is a consequence of Lemma \ref{lemma:slepian -- generalised kernel is L2}  and the compactness is due to \cite{reedMethodsModernMathematical1980}*{Theorem~VI.22}. The other properties of a) are easily derived from \eqref{K2}. 

   To obtain b),  we use the Hilbert-Schmidt Theorem (see \cite{reedMethodsModernMathematical1980}*{Chapter~VI}) which can be applied since we have just shown that \( \mathcal{K} \) is a self-adjoint and compact operator.

    To prove c) and \eqref{minmax}, we recall the min-max theorem:
    \begin{equation*}
        \lambda_n = \min_{\varphi_1, \dots, \varphi_{n-1}}   \max_{\substack{f\in L^2(\mathbb{R}^d)\\f\in [\varphi_1, \dots, \varphi_{n-1}]^\perp}} \frac{(\mathcal{K} f, f)_{L^2}}{(f,f)_{L^2}}, 
    \end{equation*}
    see for instance \cite{CheverryGuideSpectralTheory2021}*{Thm. 6.23} or  \cite{reedMethodsModernMathematical1980}*{Thm.~VI.15}. 
    Now, if  $f \in [\psi_1, \dots, \psi_{n-1}]^\perp$, we have \( f = \sum_{i \geq n} a_i \psi_i \) where \( \psi_i \) is the \( i \)-th eigenfunction of \( \mathcal{K} \).
    For such a function $f$, we get
    \begin{equation*}
        \frac{(\mathcal{K} f, f)_{L^2}}{(f,f)_{L^2}} = \frac{\sum_{i=n}^{+\infty} |a_i|^2 \lambda_i(\mathcal{K})}{\sum_{i=n}^{+\infty} |a_i|^2} \leq \lambda_n(\mathcal{K}),  
    \end{equation*}
    and the fact that equality is attained with \( f=\psi_n \) shows \eqref{minmax}. 

    The point d) follows from \cite{reedMethodsModernMathematical1980}*{Theorem~VI.22}, which states that
    \begin{equation*}
        \sum_{n=1}^\infty \lambda_n^2 < \infty.
    \end{equation*}
    This series is convergent only if \( \lambda_n = o(n^{-1/2}) \) for large \( n \).

    About the point e), it suffices to show that the inverse Fourier transform of \( |\widehat{m_F}|^2 \) is real. We have
    \begin{equation*}
        \int_{\mathbb{R}^d} |\widehat{m_F}(\xi)|^2 e^{i\xi \cdot(y-x)} d\xi
        = \int_{\mathbb{R}^d} |\widehat{m_F}(\xi)|^2 \left( \cos(\xi \cdot (y-x)) + i\sin(\xi\cdot(y-x)) \right) d\xi.
    \end{equation*}
    When \( \left| \widehat{m_F} \right|^2 \) is even, the complex part vanishes as the integral of an odd function.
    Thus, only the real part remains.
\end{proof}

Some properties of the eigenfunctions \( \{\psi_i\}_{i=1}^\infty \) can easily be obtained: 

\begin{lemma}[Symmetries]
    \label{lemma:slepian -- symmetry in eigenvectors}
    Suppose there is an orthogonal matrix \( \mathbf{S} \in \mathbb{R}^{d\times d} \) and \( \alpha\in \mathbb{C} \), \( |\alpha|=1 \), such that \( m_S \circ \mathbf{S} = \alpha m_S \) and \( \left| \widehat{m_F} \circ \mathbf{S} \right| = \left| \widehat{m_F} \right| \). 
    Then 
    $$
    \Kc (\psi \circ \mathbf{S})  =  (\Kc (\psi) ) \circ \mathbf{S}. 
    $$
    In particular, if \( \psi \) is an eigenfunction of \( \mathcal{K} \) associated to an eigenvalue \( \lambda \) of multiplicity one, then there exists \( \beta \in \mathbb{C} \), \( |\beta|=1 \), such that
    \begin{equation*}
        \psi \circ \mathbf{S} = \beta \psi.
    \end{equation*}
\end{lemma}

\begin{proof}
    It follows from straightforward computations. First of all, since \( \mathbf{S} \) is orthogonal, \( \left|\det \mathbf{S}\right| = 1 \).
    We now compute
    \begin{align*}
       (\Kc\psi)(\mathbf{S}x)
        = \int_{\mathbb{R}^d} m_S(y) \overline{m_S(\mathbf{S}x)} \mathcal{F}^{-1}\left[ |\widehat{m_F}|^2 \right](\mathbf{S}x-y) \psi(y)dy.
    \end{align*}
    The change of variables \( y = \mathbf{S}\tilde{y} \) yields
    \begin{equation*}
        (\Kc\psi)(\mathbf{S}x) 
        = \int_{\mathbb{R}^d} m_S(\mathbf{S}\tilde{y}) \overline{m_S(\mathbf{S}x)} \mathcal{F}^{-1}\left[ |\widehat{m_F}|^2 \right](\mathbf{S}x-\mathbf{S}\tilde{y}) \psi(\mathbf{S}\tilde{y})d\tilde{y},
    \end{equation*}
    where we have used \( \left|\det \mathbf{S}\right|=1 \).
    Owing to the assumption \( m_S \circ \mathbf{S} = \alpha m_S \), \( |\alpha|=1 \),
    \begin{equation*}
        (\Kc\psi)(\mathbf{S}x) 
        = \int_{\mathbb{R}^d} m_S(\tilde{y}) \overline{m_S(x)} \mathcal{F}^{-1}\left[ |\widehat{m_F}|^2 \right](\mathbf{S}x-\mathbf{S}\tilde{y}) \psi(\mathbf{S}\tilde{y})d\tilde{y}.
    \end{equation*}
    It only remains to show that \( \mathcal{F}^{-1}\left[ |\widehat{m_F}|^2 \right](\mathbf{S}x-\mathbf{S}\tilde{y}) = \mathcal{F}^{-1}\left[ |\widehat{m_F}|^2 \right](x-\tilde{y}) \). Letting \( \xi = \mathbf{S} \tilde{\xi} \),
    \begin{align*}
        \mathcal{F}^{-1}\left[ |\widehat{m_F}|^2 \right](\mathbf{S}x-\mathbf{S}\tilde{y})
        &= \frac{1}{(2\pi)^d} \int_{\mathbb{R}^d} \left| \widehat{m_F}(\xi) \right|^2 e^{i\xi \cdot \mathbf{S}(x-\tilde{y})} d\xi \\
        &= \frac{1}{(2\pi)^d} \int_{\mathbb{R}^d} \left| \widehat{m_F}(\mathbf{S}\tilde{\xi}) \right|^2 e^{i(\mathbf{S}\tilde{\xi}) \cdot \mathbf{S}(x-\tilde{y})} d\tilde{\xi}.
    \end{align*}
    We have again used the fact that \( \left| \det \mathbf{S} \right| = 1 \).
    Since \( \mathbf{S} \) is orthogonal, \( \mathbf{S}^T \mathbf{S} = I \), therefore \( (\mathbf{S} \tilde{\xi}) \cdot \mathbf{S}(x-\tilde{y}) = \tilde{\xi} \cdot (x-\tilde{y}) \) and
    \begin{equation*}
        \mathcal{F}^{-1}\left[ |\widehat{m_F}|^2 \right](\mathbf{S}x-\mathbf{S}\tilde{y})
        = \frac{1}{(2\pi)^d} \int_{\mathbb{R}^d} \left| \widehat{m_F}(\tilde{\xi}) \right|^2 e^{i\tilde{\xi} \cdot (x-\tilde{y})} d\tilde{\xi},
    \end{equation*}
    where we have used the assumption \( \left| \widehat{m_F} \circ \mathbf{S} \right| = \left| \widehat{m_F} \right| \).
    We finally obtain
    \begin{equation*}
        (\Kc\psi)(\mathbf{S}x) 
        = \int_{\mathbb{R}^d} m_S(\tilde{y}) \overline{m_S(x)} \mathcal{F}^{-1}\left[ |\widehat{m_F}|^2 \right](x-\tilde{y}) \psi(\mathbf{S}\tilde{y})d\tilde{y},
    \end{equation*}
    which is exactly \( (\mathcal{K}\psi)\circ \mathbf{S} = \mathcal{K}(\psi\circ \mathbf{S}) \).
    If \( \psi \) is an eigenfunction associated to an eigenvalue \( \lambda \) of multiplicity one, so is \( \psi \circ \mathbf{S} \).
    Therefore, they must agree up to some multiplicative constant. 
    Due to the orthogonality of \( \mathbf{S} \) they have the same \( L^2(\mathbb{R}^d) \) norm, so that constant must have modulus one.
\end{proof}

Let us consider the case where \( \lambda \) is a multiple eigenvalue of multiplicity \( p \in \mathbb{N}^* \). Write \( \phi_1, \dots, \phi_p \) the eigenfunctions of \( \mathcal{K} \) associated to \( \lambda \). 
For any \( i=1, \dots, p \), the same computations as above yield that \( \phi_i \circ \mathbf{S} \) is an eigenfunction associated to \( \lambda \). Therefore, we can only decompose 
\begin{equation*}
   \phi_i \circ \mathbf{S} = \sum_{j=1}^p b_{ij} \phi_j, \quad b_{ij}\in \mathbb{C} \text{ with } \sum_{j=1}^p |b_{ij}|=1.
\end{equation*}
This formula can have some applications, for example in presence of rotational or asymmetric invariances. Indeed, assume for instance that $\mathbf{S}^k = \mathrm{Id}$ for some $k$ (which is the case for discs or polygons with axis of symmetries in 2D or for 3D axisymmetric domains, with invariance by angular rotation of angle $\frac{2\pi}{k}$).
Then the previous relation implies that $\mathbf{B}^k = \mathrm{Id}$ which implies that the eigenfunctions can be sorted with respect to the eigenvalues $\alpha_n = e^{\frac{2i\pi n}{k}}$, $n = 1,\ldots,k$ of $\mathbf{B} = (b_{i,j})_{i, j}$. This will be easily observed in the 2D examples below (the disc and the cat-head).

\subsection{Special case of binary masks}

The situation with binary masks has some interesting properties.  For two (smooth enough) domains \( \Omega_S, \Omega_F \subset \mathbb{R}^d \) if we take $m_S(x) = \mathbf{1}_{\Omega_S}$ and $\widehat{m_F} = \mathbf{1}_{\Omega_F}$, the eigenvalue problem associated with the operator \eqref{defK} can be written 
\begin{equation}
\label{defKOm}
\mathrm{Find }\quad (\lambda,\psi),\quad \mbox{s.t.} \quad 
    \lambda \psi(x) =   \int_{\Omega_S} k(x - y)  \psi(y)  dy, \quad x \in \Omega_S, \quad 
    \displaystyle k(z) = \frac{1}{(2\pi)^d} \int_{\Omega_F} e^{i\xi\cdot z} d\xi .
\end{equation}
Note that the study of the function $k(z)$ for general domains is a difficult question, see for instance \cite{FeffermanMultiplierProblemBall1971}. 

\begin{lemma}[Translations with binary masks]
    \label{lemma: translations with binary masks}
    Let \( \Omega_S, \Omega_F \subset \mathbb{R}^d \), \( p\in \mathbb{R}^d \).
    The following equivalences hold: 
    \begin{itemize}
        \item \( (\lambda, \psi) \) is an eigenpair of the concentration operator associated to masks \( m_S = \mathbf{1}_{\Omega_S + p} \) and \( \widehat{m_F} = \mathbf{1}_{\Omega_F} \) iff \( (\lambda, \psi\circ \tau_{-p}) \) is an eigenpair of the concentration operator associated to masks \( m_S = \mathbf{1}_{\Omega_S} \) and \( \widehat{m_F} = \mathbf{1}_{\Omega_F} \);
        \item \( (\lambda, \psi) \) is an eigenpair of the concentration operator associated to masks \( m_S = \mathbf{1}_{\Omega_S} \) and \( \widehat{m_F} = \mathbf{1}_{\Omega_F+p} \) iff \( (\lambda, x\mapsto \psi(x) e^{-i p \cdot x}) \) is an eigenpair of the concentration operator associated to masks \( m_S = \mathbf{1}_{\Omega_S} \) and \( \widehat{m_F} = \mathbf{1}_{\Omega_F} \).
    \end{itemize}
\end{lemma}

\begin{proof}
    Let us start with the first claim, and consider \( (\lambda, \psi) \) an eigenpair of the concentration operator associated to masks \( m_S = \mathbf{1}_{\Omega_S+p} \) and \( \widehat{m_F} = \mathbf{1}_{\Omega_F} \). We have
    \begin{align*}
        \lambda \psi(x) &= (2\pi)^{-d} \mathbf{1}_{\Omega_S+p}(x) \int_{\Omega_S+p} \psi(y) \int_{\Omega_F} e^{i\xi \cdot (x-y)} d\xi dy \\
        \iff \lambda\psi(x) &= (2\pi)^{-d} \mathbf{1}_{\Omega_S}(\tau_p(x)) \int_{\Omega_S+p} \psi(y) \int_{\Omega_F} e^{i\xi \cdot (\tau_p(x) - \tau_p(y))} d\xi dy.
    \end{align*}
    Write \( u := \tau_p(x) \) and \( v := \tau_p(y) \), then \( dv = dy \) and 
    \begin{equation*}
        \lambda\psi(\tau_{-p}(u)) = (2\pi)^{-d} \mathbf{1}_{\Omega_S}(u) \int_{\Omega_S} \psi(\tau_{-p}(v)) \int_{\Omega_F} e^{i\xi \cdot (u-v)} d\xi dv.
    \end{equation*}
    In other words, \( (\lambda, \psi\circ \tau_{-p}) \) is an eigenpair of the concentration operator associated to masks \( m_S = \mathbf{1}_{\Omega_S} \) and \( \widehat{m_F} = \mathbf{1}_{\Omega_F} \).

    We now proceed to proving the second claim. 
    Consider \( (\lambda, \psi) \) an eigenpair of the concentration operator associated to masks \( m_S = \mathbf{1}_{\Omega_S} \) and \( \widehat{m_F} = \mathbf{1}_{\Omega_F+p} \), then
    \begin{equation*}
        \lambda \psi(x) = (2\pi)^{-d} \mathbf{1}_{\Omega_S}(x) \int_{\Omega_S} \psi(y) \int_{\Omega_F+p} e^{i\xi \cdot (x-y)} d\xi dy.
    \end{equation*}
    Write \( \xi = \zeta +p \) for \( \zeta \in \Omega_F \), then \( d\xi = d\zeta \) and
    \begin{align*}
        \lambda \psi(x) &= (2\pi)^{-d} \mathbf{1}_{\Omega_S}(x) \int_{\Omega_S} \psi(y) \int_{\Omega_F} e^{i(\zeta+p) \cdot (x-y)} d\zeta dy \\
        \iff \lambda \psi(x) e^{-i p \cdot x} &= (2\pi)^{-d} \mathbf{1}_{\Omega_S}(x) \int_{\Omega_S} \psi(y) e^{-i p\cdot y} \int_{\Omega_F} e^{i\zeta \cdot (x-y)} d\zeta dy.
    \end{align*}
    In other words, \( (\lambda, x\mapsto \psi(x) e^{-ip\cdot x}) \) is an eigenpair of the concentration operator associated to masks \( m_S = \mathbf{1}_{\Omega_S} \) and \( \widehat{m_F} = \mathbf{1}_{\Omega_F} \).
\end{proof}

The following result now generalizes the scaling invariance of the problem, implying in 1D the fact that the eigenfunctions depend only on the product of the size of the space and Fourier intervals: 
\begin{lemma}[Affine transformations with binary masks]
    \label{lemma: affine transformations with binary masks}
    Let \( \mathbf{A} \in \mathbb{R}^{d\times d} \) an invertible matrix, \( \Omega_S, \Omega_F \subset \mathbb{R}^d \), and write
    \begin{equation*}
        \mathbf{A} \Omega_F := \left\{ \mathbf{A}z : z \in \Omega_F \right\}.
    \end{equation*}
    Let \( (\lambda, \psi) \) an eigenpair of the concentration operator associated to binary masks \( m_S = \mathbf{1}_{\Omega_S} \) and \( \widehat{m_F} = \mathbf{1}_{\mathbf{A} \Omega_F} \).
    Then, \( (\lambda, \psi \circ \mathbf{A}^{-T}) \) is an eigenpair of the concentration operator associated to binary masks \( m_S = \mathbf{1}_{\mathbf{A}^T \Omega_S} \) and \( \widehat{m_F} = \mathbf{1}_{\Omega_F} \).
    The converse is also true.
\end{lemma}

\begin{proof}
    An eigenpair \( (\lambda, \psi) \) of the concentration operator associated to masks \( m_S = \mathbf{1}_{\Omega_S} \) and \( \widehat{m_F} = \mathbf{1}_{\mathbf{A}\Omega_F} \) satisfies the following equality:
    \begin{equation*}
        \lambda \psi(x) = (2\pi)^{-d} \mathbf{1}_{\Omega_S}(x) \int_{\Omega_S} \psi(y) \int_{\mathbf{A}\Omega_F} e^{i \xi \cdot (x-y)} d\xi dy.
    \end{equation*}
    Let \( \xi := \mathbf{A} \zeta \) for \( \zeta \in \Omega_F \), then
    \begin{align*}
        \lambda \psi(x) 
        &= (2\pi)^{-d} \mathbf{1}_{\Omega_S}(x) \int_{\Omega_S} \psi(y) \int_{\Omega_F} e^{i (\mathbf{A}\zeta) \cdot (x-y)} |\det \mathbf{A}| d\zeta dy \\
        &= (2\pi)^{-d} \mathbf{1}_{\Omega_S}(x) \int_{\Omega_S} \psi(y) \int_{\Omega_F} e^{i \zeta \cdot (\mathbf{A}^T (x-y))} |\det \mathbf{A}| d\zeta dy.
    \end{align*}
    Let \( u := \mathbf{A}^T x \) and \( v := \mathbf{A}^T y \), then
    \begin{equation*}
        \lambda \psi(\mathbf{A}^{-T} u) = (2\pi)^{-d} \mathbf{1}_{\mathbf{A}^T \Omega_S}(u) \int_{\mathbf{A}^T \Omega_S} \psi(\mathbf{A}^{-T} v) \int_{\Omega_F} e^{i\zeta \cdot (u-v)} d\zeta dv.
    \end{equation*}
    Hence, \( (\lambda, \psi \circ \mathbf{A}^{-T}) \) is an eigenpair of the concentration operator associated to masks \( m_S = \mathbf{1}_{\mathbf{A}^T \Omega_S} \) and \( \widehat{m_F} = \mathbf{1}_{\Omega_F} \).
\end{proof}

\section{New examples}

\subsection{Quadratic domains}
In the literature, the only examples of quadratic operator $\Pc$ commuting with $\Kc$ are the operator \eqref{Pc1} in the case of interval, as well as the case of balls obtained by Slepian. 
We give below a new general result for domains delimited by general quadrics. We define the family of \( d \)-dimensional {\em quadratic domains} as follows:
\begin{equation}
\label{Qcab}
    Q(c, a, b) := \left\{ x\in \mathbb{R}^d : \sum_{m=1}^d (x_m - c_m)^2 a_m \leq b \right\}, \quad a,c\in \mathbb{R}^d, b\in \mathbb{R}.
\end{equation}

Note that in full generality, we allow negative coefficients, but that in the case where $a_m > 0$, these domains are ellipsoidal domains. 
In the following we will assume these ellipses to be centered, {\em i.e.} \( c=0 \), to simplify the calculations. 
This is done without loss of generality using Lemma \ref{lemma: translations with binary masks}.

\begin{proposition}
    \label{proposition: differential operator ellipses}
    Consider the \( d \)-dimensional concentration problem where the space domain is restricted to \( \Omega_S := Q(0, a, b) \) and the Fourier domain to \( \Omega_F := Q(0, \alpha, \beta) \), for some \( a, \alpha \in \mathbb{R}^d \) and \( b, \beta\in \mathbb{R} \).
    Let \( \mathcal{K} \) be the concentration operator associated to masks \( m_S = \mathbf{1}_{\Omega_S} \) and \( \widehat{m_F} = \mathbf{1}_{\Omega_F} \), see \eqref{defKOm}. 
       Then, there exists a second-order differential operator \( \mathcal{P} \), self-adjoint on \( L^2(\Omega_S) \), which commutes with the concentration operator \( \mathcal{K} \). 
    It is given by
    \begin{equation}
        \label{eqn: slepian -- expression explicite operateur differentiel commutation ellipses}
        \mathcal{P}(x, \nabla_x) = \nabla_x \cdot (\mathbf{A}(x) \nabla_x) + C(x),
    \end{equation}
    where 
    \begin{equation}
    \label{eqA}
        \mathbf{A}(x) := \mathrm{diag}\left\{ \alpha_m \left( \sum_{n=1}^d a_n x_n^2 - b \right) \right\}_{m=1}^d
        \quad \text{and}\quad
        C(x) = \beta x\cdot \mathrm{diag}\left\{ a_n \right\}_{n=1}^d x.
    \end{equation}
\end{proposition}

\begin{proof}
Let $\Pc$  be a 
self-adjoint second-order differential operator of the form \eqref{eqn: slepian -- expression explicite operateur differentiel commutation ellipses}
where \( x\mapsto \mathbf{A}(x) \) is a matrix-valued real function such that  \( \mathbf{A}(x) = \mathbf{A}(x)^T \) on \( \Omega_S \), and \( x\mapsto C(x) \) is a real scalar function. Note that the condition 
    \begin{equation}
    \label{bA}
        \mathbf{A} = 0 \quad \text{ on } \partial \Omega_S = \left\{ x\in \mathbb{R}^d : \sum_{n=1}^d a_n x_n^2 = b \right\},
    \end{equation}
    ensures that \( \mathcal{P} \) is self-adjoint on \( L^2(\Omega_S) \). 
    Hence we have with the notation \eqref{defKOm}, 
    \begin{align*}
        (\mathcal{K} \mathcal{P}f)(x)
        &= \int_{\Omega_S} \left[ \mathcal{P}(y, \nabla_y)f(y) \right] k(x-y) dy = \int_{\Omega_S} f(y) \left[ \mathcal{P}(y, \nabla_y) k(x-y) \right] dy.
    \end{align*}
The commutation  relation \( \mathcal{P} \mathcal{K} = \mathcal{K}\mathcal{P} \) is thus implied by the condition     \begin{equation}
        \label{eqn: slepian1d -- commutation relation ellipse}
        \forall\, (x,y) \in \Omega_S^2,\quad 
        \mathcal{P}(x, \nabla_x) k(x-y) = \mathcal{P}(y, \nabla_y) k(x-y).
    \end{equation}
    Letting \( \varphi(x, \xi) := e^{i\xi\cdot x} \), we have
    \begin{equation*}
        k(x-y) = \frac{1}{(2\pi)^d} \int_{\Omega_F} \varphi(x, \xi) \overline{\varphi(y, \xi)} d\xi.
    \end{equation*}
    Owing to 
    \begin{equation}
        \label{eqn: slepian1d -- fourier relations ellipse}
        \nabla_x \varphi(x,\xi) = i\xi \varphi(x, \xi) \quad \text{ and } \quad x\varphi(x, \xi) = -i \nabla_{\xi} \varphi(x, \xi),
    \end{equation}
    we get 
    \begin{equation*}
        \mathcal{P}(x, \nabla_x) \varphi(x, \xi) = \mathcal{P}(-i\nabla_\xi, i\xi) \varphi(x, \xi).
    \end{equation*}
    The differential operator \( \mathcal{P} \) having real coefficients, we get
    \begin{equation*}
        \mathcal{P}(y, \nabla_y) \overline{\varphi(y, \xi)} = \overline{\mathcal{P}(y, \nabla_y) \varphi(y, \xi)}  = \overline{\mathcal{P}(-i\nabla_\xi, i\xi) \varphi(y, \xi)}.
    \end{equation*}
    The commutation relation \eqref{eqn: slepian1d -- commutation relation ellipse} then writes
    \begin{equation}
        \label{eqn: slepian -- commutation relation P on D2}
        \int_{\Omega_F} \overline{\varphi(y, \xi)} \mathcal{P}(-i\nabla_{\xi}, i\xi) \varphi(x, \xi) d\xi 
        = \int_{\Omega_F} \varphi(x, \xi) \overline{\mathcal{P}(-i\nabla_{\xi}, i\xi) \varphi(y, \xi)} d\xi.
    \end{equation}
    In other words, we want the differential operator 
    \begin{equation*}
        \mathcal{P}(-i\nabla_{\xi}, i\xi) = i\xi \cdot \left( \mathbf{A}(-i\nabla_{\xi}) (i\xi) \right) + C(-i\nabla_{\xi})
    \end{equation*}
    to be self-adjoint on \( L^2(\Omega_F) \) for the Hermitian inner product.

    Let $\mathbf{A}(x)$ be the symmetric matrix defined by the equation \eqref{eqA}, and which satisfies the boundary condition \eqref{bA}. 
    One has
    \begin{align*}
        \mathcal{P}(-i\nabla_\xi, i\xi)
        &= -\xi \cdot \mathrm{diag}\left\{ \alpha_m \left( -\sum_{n=1}^d a_n \partial_{\xi_n}^2  - b \right) \right\}_{m=1}^d \xi + C(-i\nabla_\xi) \\
        &= \sum_{m=1}^d \xi_m \alpha_m \left( \sum_{n=1}^d a_n \partial_{\xi_n}^2 \xi_m + b\xi_m \right) + C(-i\nabla_\xi)\\
        &= \sum_{m=1}^d \alpha_m \left( \sum_{n=1}^d a_n \xi_m \partial_{\xi_n}^2 \xi_m + b\xi_m^2 \right) + C(-i\nabla_\xi).
    \end{align*}
    We emphasize that \( \partial_{\xi_m} \xi_n \) is actually the operator given for \( h\in C^\infty(\mathbb{R}^d) \) by
    \begin{equation*}
        \partial_{\xi_m} \xi_n h = \partial_{\xi_m} \left( \xi_n h \right).
    \end{equation*}
    In particular, since $\partial_{\xi_m} \left( \xi_m h \right) = h + \xi_m \partial_{\xi_m} h$, we have 
    \begin{equation*}
        \partial_{\xi_m} \xi_m - \xi_m \partial_{\xi_m} = 1
    \quad \mbox{and}\quad        \partial_{\xi_m} \xi_n = \xi_n \partial_{\xi_m}, \quad m\neq n.
    \end{equation*}
    Therefore,
    \begin{align*}
        \sum_{n=1}^d a_n \xi_m \partial_{\xi_n}^2 \xi_m&= a_m \xi_m \partial_{\xi_m}^2 \xi_m + \sum_{n\neq m} a_n \xi_m \partial_{\xi_n}^2 \xi_m 
            = a_m \xi_m \partial_{\xi_m}^2 \xi_m + \sum_{n\neq m} a_n \partial_{\xi_n} \xi_m^2 \partial_{\xi_n} \\
        &= a_m \left(\partial_{\xi_m} \xi_m - 1\right) \left( \xi_m \partial_{\xi_m} + 1 \right) + \sum_{n\neq m} a_n \partial_{\xi_n} \xi_m^2 \partial_{\xi_n} \\
        &= a_m \left( \partial_{\xi_m} \xi_m^2 \partial_{\xi_m} - \xi_m \partial_{\xi_m} + \partial_{\xi_m} \xi_m - 1 \right) + \sum_{n\neq m} a_n \partial_{\xi_n} \xi_m^2 \partial_{\xi_n} \\
        &= \sum_{n=1}^d a_n \partial_{\xi_n} \xi_m^2 \partial_{\xi_n}.
    \end{align*}
Hence we have 
    \begin{align*}
        \mathcal{P}(-i\nabla_\xi, i\xi)
        &= \sum_{n=1}^d a_n \sum_{m=1}^d \alpha_m \partial_{\xi_n} \xi_m^2 \partial_{\xi_n} + b\sum_{k=1}^d \alpha_m \xi_m^2  + C(-i\nabla_\xi)
    \end{align*}
    We recognize here that
    \begin{equation*}
        \mathcal{P}(-i\nabla_\xi, i\xi)
        = \nabla_\xi \cdot \mathrm{diag}\left\{ a_n \sum_{m=1}^d \alpha_m \xi_m^2 \right\}_{n=1}^d \nabla_\xi + b \xi \cdot \mathrm{diag}\left\{ \alpha_m \right\}_{m=1}^d \xi + C(-i\nabla_\xi).
    \end{equation*}
    By choosing 
    \begin{equation*}
        C(-i\nabla_\xi) = -\beta \nabla_\xi \cdot \mathrm{diag}\left\{ a_n \right\}_{n=1}^d  = - \beta \sum_{n = 1}^d\ a_n \partial_{\xi_n}^2,
    \end{equation*}
    we obtain 
    \begin{equation*}
        \mathcal{P}(-i\nabla_\xi, i\xi)
        = \nabla_\xi \cdot \mathrm{diag}\left\{ a_n \left( \sum_{m=1}^d \alpha_m \xi_m^2 - \beta \right) \right\}_{n=1}^d \nabla_\xi + b\xi \cdot \mathrm{diag}\left\{ \alpha_m \right\}_{m=1}^d \xi.
    \end{equation*}
    It is a self-adjoint operator on \( L^2(\Omega_F) \) for the Hermitian inner product since 
    \begin{equation*}
        \mathrm{diag}\left\{ a_n \left( \sum_{m=1}^d \alpha_m \xi_m^2 - \beta \right) \right\}_{n=1}^d = 0 
        \quad \text{on}\quad 
        \partial \Omega_F = \left\{ \xi \in \mathbb{R}^d : \sum_{m=1}^d \alpha_m \xi_m^2 = \beta \right\}.
    \end{equation*}
    This shows that \eqref{eqn: slepian -- commutation relation P on D2} is satisfied, and \textit{a posteriori} that \( \mathcal{P} \) commutes with \( \mathcal{K} \).
\end{proof}

We can generalize the previous result by considering domains of the form 
\begin{equation*}
    D(\mathbf{M}, v, c) := \left\{ x\in \mathbb{R}^d : x^T \mathbf{M} x + v^T x + c \leq 0 \right\}.
\end{equation*}

\begin{theorem}
    Let \( \Omega_S = D(\mathbf{M}_S, v_s, c_S) \) and \( \Omega_F = D(\mathbf{M}_F, v_F, c_F) \), for some symmetric, diagonalizable and invertibles matrices \( \mathbf{M}_S, \mathbf{M}_F \in \mathbb{R}^{d\times d} \), vectors \( v_S, v_F \in \mathbb{R}^d \) and scalars \( c_S, c_F \in \mathbb{R} \).
    Let \( \mathcal{K} \) the concentration operator associated to masks \( m_S = \mathbf{1}_{\Omega_S} \) and \( \widehat{m_F} = \mathbf{1}_{\Omega_F} \).
    Then, there exists a second-order differential operator \( \mathcal{P} \) that commutes with \( \mathcal{K} \).

    Let \( \mathbf{M}_S = U_S \Lambda_S U_S^T \) and \( \mathbf{M}_F = U_F \Lambda_F U_F^T \), where \( U_S, U_F \in \mathbb{R}^d \) are orthogonal matrices and \( \Lambda_S, \Lambda_F \) are diagonal, and write \( a_n = (\Lambda_S)_{n, n} \), \( \alpha_n = (\Lambda_F)_{n, n} \).
    Let \( w_S = -\frac{1}{2} \Lambda_S^{-1} U_S^T v_S  \), \( w_F = -\frac{1}{2} \Lambda_F^{-1} U_F^T v_F  \), and define \( b = w_S^T \Lambda_S w_S - c_S \), \( \beta = w_F^T \Lambda_F w_F - c_F \).
    Then,
    \begin{equation*}
        \mathcal{P}(x, \nabla_x) = \nabla_x^T  U^T \mathbf{A}\left( U_S^T x + \frac{1}{2} \Lambda_S^{-1} U_S^T v_s \right) U^T \nabla_x  + C\left( U_S^T x + \frac{1}{2} \Lambda_S^{-1} U_S^T v_s \right).
    \end{equation*}
    where \begin{equation*}
        \mathbf{A}(y) := \mathrm{diag}\left\{ \alpha_m \left( \sum_{n=1}^d a_n y_n^2 - b \right) \right\}_{m=1}^d
        \quad \text{and}\quad
        C(y) := \beta y\cdot \mathrm{diag}\left\{ a_n \right\}_{n=1}^d y.
    \end{equation*}
\end{theorem}

\begin{proof}
    We simply show that it is possible via an affine change of variables to recover the case of quadratic domains of Proposition \ref{proposition: differential operator ellipses}.
    
    Let us drop the indices \( {}_S \) and \( {}_F \) since we do the exact same computations.
    We assume \( \mathbf{M} \) to be diagonalizable and symmetric, so there exist matrices \( U, \Lambda \in \mathbb{R}^{d\times d} \) such that \( \mathbf{M} = U \Lambda U^T \), where \( U \) is an orthogonal matrix and \( \Lambda \) is diagonal.
    We then have
    \begin{equation*}
        \Omega = \left\{ x\in \mathbb{R}^d : (U^T x)^T \Lambda (U^T x) + v^T x + c \leq 0 \right\}.
    \end{equation*}
    This hints for the change of variable \( \tilde{y} = U^T x \), and \( \tilde{w} :=U^T v \). Then,
    \begin{equation*}
        \Omega = \left\{ \tilde{y} \in \mathbb{R}^d : \tilde{y}^T \Lambda \tilde{y} + \tilde{w}^T \tilde{y} + c \leq 0 \right\}.
    \end{equation*}
    Since \( \mathbf{M} \) is invertible, so is \( \Lambda \). 
    Define \( w = -\frac{1}{2} \Lambda^{-1} \tilde{w} \) and \( y = \tilde{y} - w \), then 
    \begin{align*}
        y^T \Lambda y = (\tilde{y}-w)^T \Lambda (\tilde{y}-w)
        &= \tilde{y}^T \Lambda \tilde{y} - 2 w^T \Lambda \tilde{y} + w^T \Lambda w \\
        &= \tilde{y}^T \Lambda \tilde{y} + \tilde{w}^T \tilde{y} + w^T \Lambda w,  
    \end{align*}
    and we get
    \begin{equation*}
        \Omega = \left\{ y \in \mathbb{R}^d : y^T \Lambda y - w^T \Lambda w + c \leq 0  \right\}.
    \end{equation*}
    We observe that \( \Omega \) is of the form \eqref{Qcab}. 

    By doing similar changes of variables for the space and Fourier domains \( \Omega_S \) and \( \Omega_F \), we can apply Proposition \ref{proposition: differential operator ellipses} and obtain a commuting differential operator \( \mathcal{P} \) in variables \( (y, \nabla_y) \).
    
    In order to obtain the explicit expression of \( \mathcal{P} \) in variables \( (x, \nabla_x) \), we have to define a few quantities.
    Let \( a_n := (\Lambda_S)_{n, n} \), \( \alpha_n := (\Lambda_F)_{n, n} \), \( b := w_S^T \Lambda_S w_S - c_S \), \( \beta := w_F^T \Lambda_F w_F - c_F \), and recall \( y = U_S^T x + \frac{1}{2} \Lambda_S^{-1} U_S^T v_s \). Also, define
    \begin{equation*}
        \mathbf{A}(y) := \mathrm{diag}\left\{ \alpha_m \left( \sum_{n=1}^d a_n y_n^2 - b \right) \right\}_{m=1}^d
        \quad \text{and}\quad
        C(y) := \beta y\cdot \mathrm{diag}\left\{ a_n \right\}_{n=1}^d y.
    \end{equation*}
    Then
    \begin{equation}
        \mathcal{P}(y, \nabla_y) = \nabla_y \cdot (\mathbf{A}(y) \nabla_y) + C(y),
    \end{equation}

    We have \( \nabla_y = U^T \nabla_x \), thus
    \begin{equation*}
        \mathcal{P}(x, \nabla_x) = \nabla_x^T  U^T \mathbf{A}\left( U_S^T x + \frac{1}{2} \Lambda_S^{-1} U_S^T v_s \right) U^T \nabla_x  + C\left( U_S^T x + \frac{1}{2} \Lambda_S^{-1} U_S^T v_s \right).
    \end{equation*}
\end{proof}

By Lemma \ref{lemma:slepian -- eigenfunctions commuting operators}, this commutation property allows us to look for eigenfunctions of \( \mathcal{P} \) in order to know the eigenfunctions of \( \mathcal{K} \).

\begin{lemma}
    \label{lemma:slepian -- eigenfunctions commuting operators}
    Let \( d \geq 1 \), and \( \mathcal{Q}, \mathcal{R}: L^2(\mathbb{R}^d) \to L^2(\mathbb{R}^d) \) two commuting operators acting on \( L^2(\mathbb{R}^d) \). In other words, \( \mathcal{Q}\mathcal{R} = \mathcal{R}\mathcal{Q} \).
    Suppose that each eigenfunction \( \varphi_i \) of \( \mathcal{Q} \) is associated to an eigenvalue \( \kappa_i \) of multiplicity one, and that \( \left\{ \varphi_i \right\}_{i\in \mathbb{N}} \) is a complete family in \( L^2(\mathbb{R}^d) \).
    Then \( \mathcal{Q} \) and \( \mathcal{R} \) have the same eigenfunctions.
\end{lemma}

\begin{proof}
    Using the commutation relation between \( \mathcal{Q} \) and \( \mathcal{R} \), one obtains
    \begin{equation*}
        \mathcal{Q}\varphi_i = \kappa_i \varphi_i \implies \mathcal{R}\mathcal{Q}\varphi_i = \kappa_i \mathcal{R}\varphi_i \implies \mathcal{Q}\mathcal{R}\varphi_i = \kappa_i \mathcal{R}\varphi_i.
    \end{equation*}
    This means that \( \mathcal{R}\varphi_i \) is also an eigenfunction of \( \mathcal{Q} \), and since \( \kappa_i \) is an eigenvalue of multiplicity one we must have \( \mathcal{R}\varphi_i = c \varphi_i \) for some constant \( c \in \mathbb{C} \).
    In other words, all eigenfunctions of \( \mathcal{Q} \) are eigenfunctions of \( \mathcal{R} \).
    Moreover, since the eigenfunctions of \( \mathcal{Q} \) form a complete family of \( L^2(\mathbb{R}^d) \), we deduce that the eigenfunctions of \( \mathcal{R} \) are exactly the eigenfunctions of \( \mathcal{Q} \).
\end{proof}

\subsection{A splitting approach}

Assume that $1 \geq m_S> 0$ and $1 \geq \widehat{m_F} > 0$. Then there exists functions $V(x) \geq 0$ and $H (\xi) \geq 0$ such that 
$$
m_S(x) = e^{-\frac12 V(x)} \quad \mbox{and} \quad \widehat{m_F}(\xi) = e^{-\frac12 H(\xi)}.  
$$
Owing the the fact that $\xi \widehat f (\xi) = -i \widehat{\partial_x f} (\xi)$, the operator $\Kc$ can be expressed under the splitting form 
\begin{equation}
\label{Ksplit}
\Kc = e^{-\frac12 V(x)}e^{- H(- i \partial_x)} e^{-\frac12 V(x)}
\end{equation}
and we recognize a {\em Strang splitting} decomposition of the pseudo-differential operator $H(-i\partial_x) + V(x)$ (see \cites{hairerGeometricNumericalIntegration2006,JahnkeErrorBoundsExponential2000}). In particular, we can write formally the Baker-Campbell-Hausdorff (BCH) formula \cites{bakerAlternantsContinuousGroups1905,hairerGeometricNumericalIntegration2006}
$$
e^{-\frac12 V(x)}e^{- H(- i \partial_x)} e^{-\frac12 V(x)} = e^{- Z(x,-i\partial_x) }
$$
where 
\begin{equation}
\label{BCH}
Z(x,-i\partial_x) = H + V + \frac{1}{12}[H,[H,V]] - \frac{1}{24}[V,[V,H]] + \sum_{k\geq 2}{Z_{2k+1}}
\end{equation}
where the $Z_{2k+1}$ are made of nested commutators between the operators $H$ and $V$. Note that in general, if $H$ and $V$ are polynomials, then the operators $Z_{2k+1}$ can be expressed as polynomials of higher degrees making the previous series non convergent and a general singularly perturbed problem. Before giving more precise example, we first give a case of convergence: 

\begin{theorem}
    Let $\alpha,\beta > 0$,  
    $m_S = e^{-\frac{\alpha}{2} x^2}$ and $\widehat{m_F} = e^{- \frac{\beta}{2} \xi^2}$. Then we have 
    \begin{equation}
        \label{Kgauss}
        \Kc = e^{- \frac12 \alpha x^2} e^{  \beta\Delta } e^{- \frac12 \alpha x^2}
        = e^{ - \mathrm{argsh}(\sqrt{\alpha\beta})\left[\sqrt{\frac{\alpha(1 + \alpha\beta)}{\beta}} x^2 - \sqrt{\frac{\beta}{\alpha(1 + \alpha\beta)}} \Delta\right]}. 
    \end{equation}
    The eigenpairs \( (\lambda_n, \psi_n)_n \) of the operator $\Kc$ are given by 
    \begin{align*}
        \psi_n &= \left( \frac{\alpha(1 + \alpha \beta)}{\beta}  \right)^{\frac{1}{8} }
        \varphi_n \left( \left( \frac{\alpha(1 + \alpha \beta)}{\beta} \right)^{\frac14} x\right)\\
        \lambda_n &= e^{- \mathrm{argsh}(\sqrt{\alpha \beta}) (2n +1)}
    \end{align*}
    where $\varphi_n$ are Hermite functions. 
\end{theorem}
\begin{proof}
We start with the following formula (see \cite{alphonsePolarDecompositionSemigroups2023}), for $z >0$, 
$$
e^{- \frac12 \tanh(z) x^2} e^{ \frac12 \sinh(2z)\Delta } e^{- \frac12 \tanh(z) x^2}
= e^{ - z ( x^2 - \Delta)},
$$
from which we deduce by a scaling by $\sqrt{\cosh z}$ in $x$:  
$$
e^{- \frac12 \sinh(z) x^2} e^{  \sinh(z)\Delta} e^{- \frac12 \sinh(z) x^2}
= e^{ - z\left( \cosh(z) x^2 - \frac{1}{\cosh(z)}\Delta\right)}
$$
and hence for all $c > 0$
$$
e^{- \frac12 c x^2} e^{  c\Delta } e^{- \frac12 c x^2}
= e^{ - \mathrm{argsh}(c)\left[\sqrt{1 + c^2}   x^2 - \frac{1}{\sqrt{1 + c^2}}\Delta\right]}. 
$$
Now by scaling again we obtain 
$$
e^{- \frac12 c \lambda^2 x^2} e^{  \frac{c}{\lambda^{2}}\Delta } e^{- \frac12 c \lambda^2 x^2}
= e^{ - \mathrm{argsh}(c)\left[\sqrt{1 + c^2} \lambda^2  x^2 - \frac{1}{\lambda^2\sqrt{1 + c^2}}\Delta\right]}. 
$$
and by taking $\lambda^2 = \sqrt{\frac{\alpha}{\beta}}$ and $c = \sqrt{\alpha\beta}$ we obtain \eqref{Kgauss}. Remark that this formula implies that the operator 
$$
\Hc = -\frac1{\mu^2} \Delta + \mu^2 x^2, \qquad \mu^2 = \sqrt{\frac{\alpha(1 + \alpha\beta)}{\beta}}
$$
commutes with $\Kc$. Moreover, the spectrum of the operator $\mathcal{H}$ is given by the normalized functions
$\sqrt{\mu}\varphi_{n}(\mu x)$ with eigenvalues $2n +1$, where the $\varphi_n$ are the normalized Hermite functions. This implies the result. 
\end{proof}
\begin{remark}
Using the framework of \cite{alphonsePolarDecompositionSemigroups2023}, this result can be easily extended to more general masks with quadratic function $H(\xi)$ and $V(x)$, and in higher dimension. The details of the exact formula are left to the reader. 
\end{remark}

We conclude this section by giving another consequence of Formula \eqref{Ksplit} which is a possible construction of quasimodes for the operator $\Kc$ for masks that are close the the identity. Again, we give only a simple example of application. Let us define the spaces 
$$
\Vc^s = \{ f \in L^2(\R^2), \quad | \langle x \rangle^s f \in L^2(\R^d)  \quad\mbox{and}\quad
\langle \xi \rangle^s \widehat f \in L^2(\R^2\}, 
$$
where $\langle y \rangle^2 = 1 + |y|^2$ for $y \in \R^d$. 
\begin{theorem} Assume that $m_S(x)= e^{- \frac12 \varepsilon V(x)}$ and $\widehat{m_F}(\xi) = e^{-\frac12 \varepsilon H(\xi)}$ with $\varepsilon > 0$, and $H(\xi)\geq 0$, $V(x)\geq 0$ two given smooth functions with polynomial growth {\em i.e.} there exists $C$ and $r,p \geq 0$ such that 
$$
|V(x)| \leq C \langle x \rangle^r, \quad |H(\xi)| \leq C \langle \xi \rangle^p,\quad \forall\, (x,\xi) \in \R^{2d}. 
$$
Assume that $(\omega,\psi) \in \R \times L^2(\R^d)$ is an eigenpair of the self adjoint operator  
$$
T = H(-i \partial_x) + V(x), \quad \mbox{\em i.e.} \quad T \psi = \omega \psi. 
$$ 
Assume moreover that for all $s \geq 0$, $\psi \in \Vc^s$, 
Then there exists $C$ and $\varepsilon_0$ such that for $\varepsilon \leq \varepsilon_0$ we have  
$$
\Norm{\Kc \psi - e^{- \varepsilon \omega} \psi}{L^2} \leq C \varepsilon^3.
$$
\end{theorem}
\begin{proof}
The proof is a consequence of the classical bounds for operator splitting as in \cite{JahnkeErrorBoundsExponential2000}. Indeed, we have for all $\varphi$ smooth enough 
$$
\Norm{e^{-\frac12\varepsilon V}e^{-\varepsilon H} e^{-\varepsilon\frac12 V} \varphi - e^{- \varepsilon T} \varphi }{L^2} \leq C \varepsilon^3 \Norm{\varphi}{\Vc^{N(p,r)}}
$$
where the exponent $N(p,r)$ depends on commutator bounds of the operators $[H,[H,V]]$ and $[V,[V,H]]$ which depend themselves of $p$, $r$ and the algebraic structure of the commutators (see formula (2.4) of \cite{JahnkeErrorBoundsExponential2000}). Note that the fact that $H$, $V$ and thus $T$ are positive operators allow to define the semi group actions $e^{- t H}$, $e^{-t V}$ and $e^{-t T}$ in the Sobolev spaces $\Vc^s$. By applying the previous estimate to the function $\psi$ we obtain the result. 
\end{proof}
\begin{remark}
We could imagine that the previous theorem could be extended to the construction of quasi-mode of arbitrary order ({\em i.e.} with a precision $\mathcal{O}(\varepsilon^N)$ for all $N$), but this would require more elaborated mathematical techniques. Works in this direction can be found  in \cite{DebusscheWeakBackwardError2012} for a related analysis in the same ``parabolic'' situation, and in \cite{faouGeometricNumericalIntegration2012} for a more general treatment of the BCH formula for PDEs. 
\end{remark}
\section{The varying masks algorithm}

We describe now a new algorithm for computing the spectrum of $\Kc$, which gives very good and promising results in known unstable situations. Note however that no rigorous analysis of this numerical method is performed in this paper. Numerical examples will be given in the next section. 

\subsection{Discretization of the generalized concentration operator}
\label{subsect: discretized concentration operator}

For simplicity we assume that the space and Fourier masks are compactly supported in the intervals 
 \( [-1, 1]^d \) and \( [-\pi, \pi]^d \) respectively. 
The space domain is discretized using the uniform midpoint rule, and the stepsize is the same for all dimensions. If we let \( N \) the number of discretization points in each dimension, then the stepsize is given by \( \Delta x = \frac{2}{N}  \).
Nodes on this grid write \( x^{(k)} = -1 + \left( k + \frac{1}{2} \right) \Delta x \) for \( k \in [\![ 0, N-1 ]\!]^d \).

\bigskip

We are interested in discretizing the concentration operator \eqref{defK}, which can be written in this case
\begin{equation*}
    (\mathcal{K}f)(x) = \int_{[-1, 1]^d} f(y) k(y,x) dy, \quad x\in [-1, 1]^d
\end{equation*}
where 
\begin{equation*}
    k(y,x) = m_S(y) \overline{m_S(x)} \mathcal{F}^{-1}\left[ \left| \widehat{m_F}(\xi) \right|^2 \right](x-y).
\end{equation*}

For every node \( x^{(j)} \) in the space grid, we obtain 
\begin{equation*}
    \int_{[-1, 1]^d} f(y) k(y,x^{(j)})dy 
    \approx (\Delta x)^d \sum_{k\in [\![ 0, N-1 ]\!]^d} f(y^{(k)}) k(y^{(k)}, x^{(j)}),
\end{equation*}
where the variable \( y \) is assumed to be discretized exactly as the \( x \) variable.
It remains to compute \( k(y^{(k)}, x^{(j)}) \) by approximating 
\begin{equation*}
    \mathcal{F}^{-1}\left[ \left| \widehat{m_F}(\xi) \right|^2 \right](x^{(j)}-y^{(k)}).
\end{equation*}
Note that the difference \( z^{(j, k)} := x^{(j)} - y^{(k)} \) belongs to a uniform discretization centered around origin, with the same stepsize \( \Delta x = \frac{2}{N}  \) and with \( 2N-1 \) uniform points in each dimension.

\bigskip

We want to approximate the inverse Fourier transform by an inverse discrete Fourier transform, let us show that this approximation holds.

Conceptually, since \( \widehat{m_F} \) is assumed to be compactly supported on \( [-\pi, \pi]^d \), we can consider the restriction \( \widehat{m_F}|_{[-\pi, \pi]^d} \) and extend it periodically to \( \mathbb{R}^d \). 
Therefore, we are looking at the inverse Fourier transform of a periodic function, which is a discrete function.
Moreover, since we are discretizing the inverse Fourier transform, the function \( \widehat{m_F} \) is only evaluated on a grid over \( [-\pi, \pi]^d \).
For simplicity, we assume this grid to be uniform with \( M \) points in each dimension, so the resulting function is not only discrete but also periodic with period \( M \).
The Fourier integral will be discretized using the midpoint quadrature rule, so the stepsize is given by \( \Delta \xi = \frac{2\pi}{M} \), and a node of this grid discretization writes \( \xi^{(l)} = -\pi + \left( l + \frac{1}{2} \right) \Delta \xi \) for \( l\in [\![ 0, M-1 ]\!]^d \).
For \( u\in [\![ 0, M-1 ]\!]^d \), we get
\begin{align*}
    \mathcal{F}^{-1}\left[ \left| \widehat{m_F}(\xi) \right|^2 \right](u)
    &= \frac{1}{(2\pi)^d} \int_{[-\pi,\pi]^d} \left| \widehat{m_F}(\xi) \right|^2 e^{i\xi \cdot u} d\xi \\
    &\approx \frac{(\Delta \xi)^d}{(2\pi)^d} \sum_{l\in [\![0, M-1 ]\!]^d} \left| \widehat{m_F}(\xi^{(l)}) \right|^2 e^{i\xi^{(l)} \cdot u} d\xi \\
    &\approx \frac{e^{i\left(-\pi + \frac{\Delta\xi}{2}\right) \mathds{1}\cdot u}}{M^d} \sum_{l\in [\![0, M-1 ]\!]^d} \left| \widehat{m_F}(\xi^{(l)}) \right|^2 e^{i \frac{2\pi}{M} l \cdot u} d\xi,
\end{align*}
where \( \mathds{1} = (1, \dots, 1) \).

We recognize here the inverse Discrete Fourier transform (IDFT) of the function \( \left| \widehat{m_F} \right|^2 \) evaluated on the \( \xi \)-grid.
Note that we actually want to compute this IDFT on the grid \( x-y \) which has \( 2N-1 \) points, so we choose \( M=2N-1 \). 
Moreover, \( u \) is the multi-index of some node in the \( x-y \) grid, which means that if we are at the node \( x^{(j)}-y^{(k)} \) for some \( j,k \in [\![ 0, N-1 ]\!]^d \), then \( u = j-k \).

We obtain
\begin{align*}
    (\mathcal{K}f)(x^{(j)})
    &\approx (\Delta x)^d \sum_{k\in [\![ 0, N-1 ]\!]^d} f(y^{(k)}) m_S(y^{(k)}) \overline{m_S(x^{(j)})} \frac{e^{i\left(-\pi + \frac{\Delta\xi}{2}\right) \mathds{1}\cdot (j-k)}}{(2N-1)^d} \times \\
    &\hspace{3cm}\sum_{l\in [\![0, 2N-2 ]\!]^d} \left| \widehat{m_F}(\xi^{(l)}) \right|^2 e^{i\frac{2\pi}{2N-1} l \cdot (j-k)}.
\end{align*}

For numerical reasons, it is advisable not to manipulate quantities that are too small. 
Therefore, we consider a normalized eigenproblem, normalized by \( (\Delta x)^d \).
Due to the scaling properties of the operator, this normalization is equivalent to work on a cube of order $(\Delta x)^{-{d}}[-1,1]^d$ but does not change the original problem.  
To simplify further, we can multiply the \( j \)-th component of \( v \) by \( e^{i (-\pi + \frac{\Delta \xi}{2} ) \mathds{1}\cdot j} \), where \( v \) is the grid discretization of a function \( f \).
We then obtain the following eigenproblem:
\begin{equation*}
    \lambda v = \mathbf{K} v,
\end{equation*}
where the matrix \( \mathbf{K} \) is the concentration matrix, and it is given by
\begin{equation}
    \label{eqn: slepian -- matrice de concentration}
    \begin{aligned}
        \mathbf{K}_{j, k} 
        &= \frac{m_S(y^{(k)}) \overline{m_S(x^{(j)})}}{(2N-1)^d} \sum_{l\in [\![0, 2N-2 ]\!]^d} \left| \widehat{m_F}(\xi^{(l)}) \right|^2 e^{i\frac{2\pi}{2N-1} l \cdot (j-k)},
    \end{aligned}
\end{equation}
for \( j, k\in [\![ 0, N-1 ]\!]^d \).
We recognize the Discrete Fourier transform, and we emphasize that the matrix defined by \eqref{eqn: slepian -- matrice de concentration} allows to compute the matrix \( \mathbf{K} \) efficiently using the Fast Fourier Transform.
Moreover, \( \mathbf{K} \) can be written as a block matrix.
For simplicity, we give details below in the two-dimensional case:
\begin{equation}
    \label{eqn: slepian -- bfK block-toeplitz structure}
    \mathbf{K} = \begin{pmatrix}
        \mathbf{K}^{(0, 0)} & \mathbf{K}^{(0, 1)} & \cdots & \mathbf{K}^{(0, N-2)} & \mathbf{K}^{(0, N-1)} \\
        \mathbf{K}^{(1, 0)} & \mathbf{K}^{(1, 1)} & \cdots & \mathbf{K}^{(1, N-2)} & \mathbf{K}^{(1, N-1)} \\
        \vdots & \vdots & \cdots & \vdots & \vdots \\
        \mathbf{K}^{(N-2, 0)} & \mathbf{K}^{(N-2, 1)} & \cdots & \mathbf{K}^{(N-2, N-2)} & \mathbf{K}^{(N-2, N-1)} \\
        \mathbf{K}^{(N-1, 0)} & \mathbf{K}^{(N-1, 1)} & \cdots & \mathbf{K}^{(N-1, N-2)} & \mathbf{K}^{(N-1, N-1)}
    \end{pmatrix},
\end{equation}
where each block \( \mathbf{K}^{(r, c)} \) is defined component-wise by
\begin{equation*}
    [\mathbf{K}^{(r,c)}]_{m, n} := \mathbf{K}_{j=(r, m), k=(c, n)}.
\end{equation*}

Let us now explain why this indexing convention is particularly efficient.
In \eqref{eqn: slepian -- matrice de concentration}, we can write \( j - k = (j_1-k_1, j_2-k_2) = (r-c, m-n)\). 
This means that, along each diagonal of \( \mathbf{K}^{(r,c)} \), the value of \( (j-k) \) is constant: indeed, \( r-c \) is constant in \( \mathbf{K}^{(r,c)} \), and \( m-n \) is constant along each diagonal of \( \mathbf{K}^{(r, c)} \).
In other words, each submatrix \( \mathbf{K}^{(r,c)} \) is a Toeplitz matrix, multiplied row-wise by the function \( \overline{m_S} \) and column-wise by \( m_S \).
The Toeplitz nature of each block \( \mathbf{K}^{(r,c)} \) allows for efficient computational storage and complexity. 

This two-dimensional discussion easily generalizes to the multi-dimensional case when only the last index varies within each block \( \mathbf{K}^{(r_1, \dots, r_{d-1})} \) of \( \mathbf{K} \). 
Therefore, each block can be expressed as a Toeplitz matrix and component-wise multiplications.

We have the following easy result, which is more or less a discrete version of Proposition \ref{prop:slepian -- properties of the operator K}:
\begin{proposition}
    \label{proposition:slepian -- properties of the matrix K}
    The matrix \( \mathbf{K} \) enjoys the following properties:
    \begin{enumerate}[label*=\emph{\alph*)}]
        \item Hermitian character: \( \mathbf{K}^* = \mathbf{K} \).
        \item Structure: \( \mathbf{K} = D^* B D \), with \( B \) a block matrix where each block is Toeplitz, and \( D \) is a diagonal matrix.
        \item Its eigenvalues are real.
        \item Its eigenvectors form an unitary basis of \( \mathbb{C}^{N_1 \cdots N_d} \).
    \end{enumerate}
\end{proposition}

\begin{proof}
    For the first point, we use \eqref{eqn: slepian -- matrice de concentration} combined with the fact that \( \left| \widehat{m_F} \right| \) is a real-valued function.

    The second point has already been mentioned earlier: the diagonal matrix \( D \) follows from \eqref{eqn: slepian -- matrice de concentration}, and it corresponds to the component-wise multiplication by the function \( \overline{m_S(x^{(j)})} \) for each row \( j \) and by \( m_S(y^{(k)}) \) for each column \( k \) of \( \mathbf{K} \).
    Hence \( \mathbf{K} \) is of the form \( D^* B D \), where \( B \) is some matrix.
    The block nature of \( B \), where each block is a Toeplitz matrix, follows from \eqref{eqn: slepian -- bfK block-toeplitz structure}.

    The third point is due to the Hermitian character of \( \mathbf{K} \).
    The fourth point is a classical result in linear algebra: for any normal matrix, there exists an orthonormal basis of eigenvectors. See for instance \cite{axlerLinearAlgebraDone2024}*{Theorem~7.31}.
\end{proof}

We end this section with a remark concerning notation: we are interested in eigenpairs of the matrix \( \mathbf{K} \) of finite dimension. With a slight abuse of notation, we will denote \( \psi_i \) the eigenvectors, which is the same notation as used for the eigenfunctions of \( \mathcal{K} \).
Whether we are talking about an eigenvector or an eigenfunction will always be clear from the context: if we are talking about the continuous concentration operator \( \mathcal{K} \), \( \psi_i \) will denote an eigenfunction, and if we are talking about the discretized version of \( \mathcal{K} \) (i.e. the matrix \( \mathbf{K} \)), then \( \psi_i \) will denote an eigenvector.

\subsection{Approximating eigenvectors}
\label{sect:slepian -- Approximating eigenvectors}

We focus now on the case of binary compact filters, i.e. 
\begin{equation*}
    m_S = \mathbf{1}_{\Omega_S} \quad \text{ and }\quad \widehat{m_F} = \mathbf{1}_{\Omega_F},
\end{equation*}
for two compact subsets \( \Omega_S, \Omega_F \subset \mathbb{R}^d \). 

Let \( \Omega_S:\mathbb{R}_+ \to [0, 1] \) and \( \Omega_F:\mathbb{R}_+ \to [0, 1] \) two set-valued functions that depend on a parameter \( \varepsilon \). 
Note the slight abuse of notation, where \( \Omega_i(\cdot) \) denotes a set-valued function while \( \Omega_i \) denotes a subset of \( \mathbb{R}^d \), \( i\in \{S, F\} \).
The functions \( \Omega_i(\cdot) \) are chosen such that \( \Omega_i(0) = \Omega_i \), and such that \( \Omega_i(\varepsilon) \) reduces to a set \( Z_i \) of null measure as \( \varepsilon\to\infty \), \( i\in \{S, F\} \).
To these set-valued functions we associate \emph{modified masks} \( m_S(\varepsilon, \cdot) := \mathbf{1}_{\Omega_S(\varepsilon)} \) and \( \widehat{m_F}(\varepsilon, \cdot) = \mathbf{1}_{\Omega_F(\varepsilon)} \). 
We may use the shorthands \( m_S = m_S(0, \cdot) \) and \( \widehat{m_F} = \widehat{m_F}(0, \cdot) \).

We denote \( \mathcal{K}(\varepsilon) \) the spectral concentration operator with space mask \( m_S(\varepsilon, \cdot) \) and Fourier mask \( \widehat{m_F}(\varepsilon, \cdot) \), and \( \mathbf{K}(\varepsilon) \) its discretization as described in Section \ref{subsect: discretized concentration operator}. 
We may use the shorthand \( \mathbf{K} = \mathbf{K}(0) \).
Moreover, we write \( \big( \lambda_n(\varepsilon), \psi_n(\varepsilon) \big) \) the \( n \)-th eigenpair of the modified concentration matrix \( \mathbf{K}(\varepsilon) \).

We define the concentration ratio of a vector similarly to the continuous setting:
\begin{equation*}
    \nu(v) = \frac{v^* \mathbf{K} v}{v^* v}. 
\end{equation*}
Note that it is computed with respect to the initial concentration matrix \( \mathbf{K}(0) \).

We now make an assumption, which will be crucial in the following. It has always been observed to hold in practice during our experiments, so we believe this assumption does not impose too much restriction.

\begin{assumption}
    \label{assumption: behavior of eigenvalues depending on varepsilon}
    It is assumed that, as \( \varepsilon\to 0 \), the eigenvalues corresponding to the concentration operator with masks \( m_S(\varepsilon, \cdot) \) in space and \( \widehat{m_F}(\varepsilon, \cdot) \) in Fourier are such that the first eigenvalue \( \lambda_1(\varepsilon) \) reaches \( \lambda_1(0) \) before \( \lambda_2(\varepsilon) \) reaches \( \lambda_2(0) \), and so on for the next eigenvalues \( \lambda_n(\varepsilon) \). 
    In other words, we assume that eigenvalues corresponding to the modified masks are all (numerically!) distinct for \( \varepsilon>0 \), and that the order \( \lambda_i(\varepsilon) > \lambda_{i+1}(\varepsilon) \) is preserved for all \( \varepsilon \).
\end{assumption}

To illustrate this assumption we go back to the one-dimensional historical example by Slepian, but instead of considering \( m_S = \mathbf{1}_{[-1, 1]} \) and \( \widehat{m_F} = \mathbf{1}_{[-c, c]} \), we consider their modified versions \( m_S(\varepsilon, \cdot) = \mathbf{1}_{[-\mu(\varepsilon), \mu(\varepsilon)]} \) and \( \widehat{m_F}(\varepsilon, \cdot) = \mathbf{1}_{[-c \mu(\varepsilon), c \mu(\varepsilon)]} \), where 
\begin{equation}
    \label{eqn:slepian1d -- mu eps}
    \mu(\varepsilon) := \frac{1}{(1+\varepsilon^4)^{1/4}}. 
\end{equation}

There is no numerical issue in looking for the eigenvalues, even if they are close to each other. Issues arise when we look for eigenvectors.
We give in Figure \ref{fig: 1d -- evolution of eigenvalues depending on varepsilon} the 30 first eigenvalues obtained for several values of \( \varepsilon \).
In Figure \ref{fig: 1d -- eigvecs depending on varepsilon compared to dpss}, we give the first 16 eigenvectors obtained with \( \varepsilon=0 \) (solid blue curve), with \( \varepsilon=2 \) (dashed orange curve), as well as the exact ones, in the case of a concentration matrix corresponding to \( m_S(\varepsilon, \cdot) = \mathbf{1}_{[-\mu(\varepsilon), \mu(\varepsilon)]} \) and \( \widehat{m_F}(\varepsilon, \cdot) = \mathbf{1}_{[-0.1\cdot 2\pi\mu(\varepsilon), 0.1\cdot 2\pi\mu(\varepsilon)]} \).
The exact eigenvectors in this particular case are known as the eigenvectors of a tridiagonal matrix and they are called Discrete Prolate Spheroidal Sequences (DPSS), see e.g. \cite{slepianProlateSpheroidalWave1978}.

We can observe that \( \varepsilon=0 \) yields eigenvectors that are not symmetric (neither even nor odd), while \( \varepsilon=2 \) does. 
It is known in this particular situation that the eigenvectors are symmetric.
Moreover, \( \varepsilon=2 \) yields some vectors that may happen be close to the desired ones: the first vectors are very close to the exact eigenvectors, while the next vectors are very different from the exact eigenvectors.

\begin{figure}
    \centering
    \includegraphics[width=\linewidth]{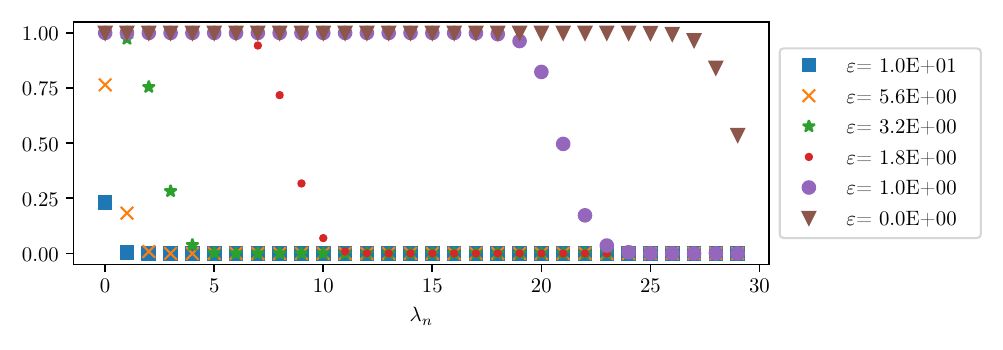}
    \caption{30 first eigenvalues of the matrix \( \mathbf{K}(\varepsilon) \), with \( N=150 \) points of discretization, \( \Omega = 0.1\cdot 2\pi \).}
    \label{fig: 1d -- evolution of eigenvalues depending on varepsilon}
\end{figure}

\begin{figure}
    \centering
    \includegraphics[width=\linewidth]{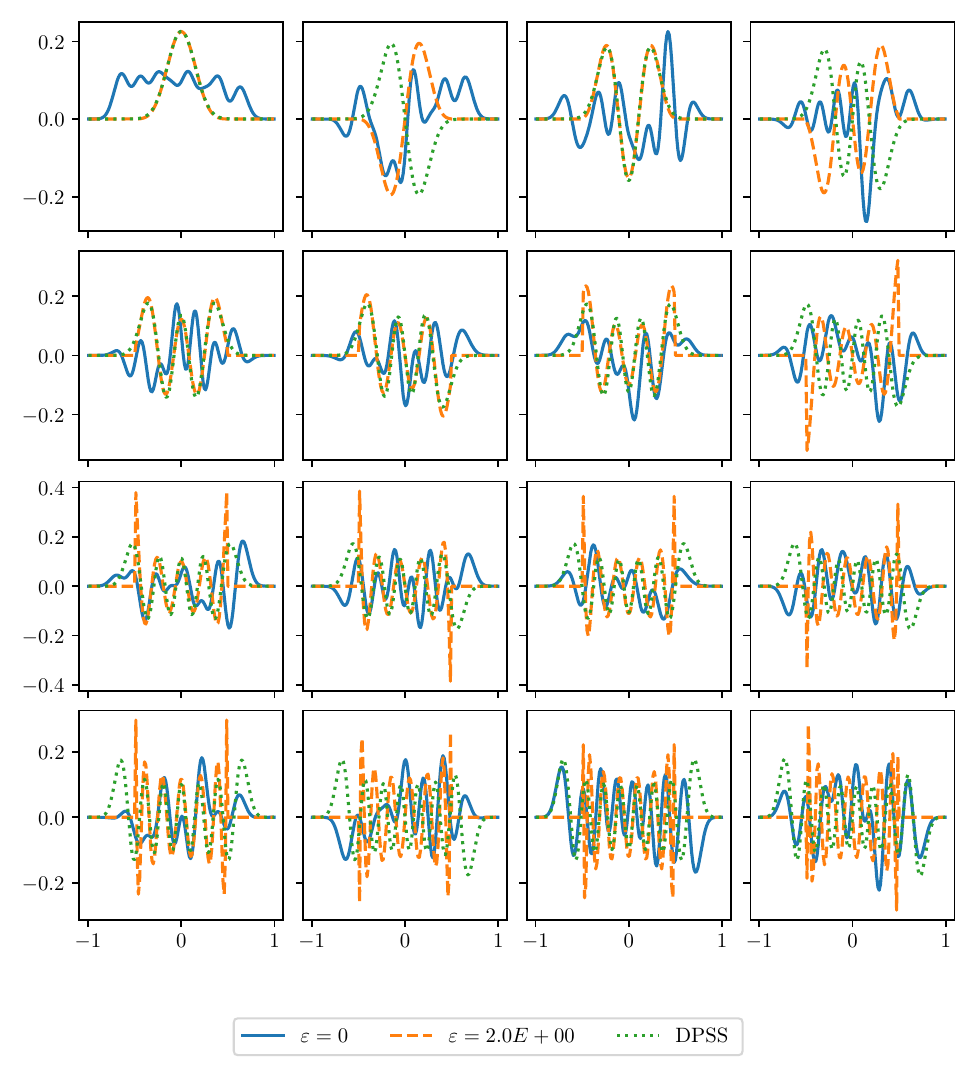}
    \caption{First sixteen eigenvectors obtained for the initial discretized concentration problem (\( \varepsilon=0 \), solid blue), for the modified discretized concentration problem (\( \varepsilon=2 \), orange dashes), and the exact eigenvectors given by the DPSS (green dots).}
    \label{fig: 1d -- eigvecs depending on varepsilon compared to dpss}
\end{figure}

The key takeaway from this simple experiment is that when the concentration problem is scaled down (or ``shrinked''), the first eigenvectors are the same but the next ones have a smaller concentration ratio.
This means that it is now possible and easy to obtain approximate eigenvectors, by looking at different values of \( \varepsilon>0 \).

This motivates the following idea for an approximate eigendecomposition of the spectral concentration matrix \( \mathbf{K} \): start from very narrow modified masks (\textit{i.e.} \( \varepsilon \) large), then it is easy to obtain the first eigenvector \( \psi_1(\varepsilon) \) of \( \mathbf{K}(\varepsilon) \) since the gap between the first and second eigenvalues is large according to Assumption \ref{assumption: behavior of eigenvalues depending on varepsilon}. 
However, this vector \( \psi_1(\varepsilon) \) is probably not a satisfying approximation of the first eigenvector \( \psi_1(0) \) of \( \mathbf{K}(0) \), in the sense that its concentration ratio may not be close to the exact first eigenvalue \( \lambda_1(0) \) of \( \mathbf{K}(0) \).
This is the situation occuring for \( \varepsilon=10 \) or \( \varepsilon\approx 5.6 \) in Figure \ref{fig: 1d -- evolution of eigenvalues depending on varepsilon}.
When this happens, we can take \( \varepsilon \) smaller and check again if \( \psi_1(\varepsilon) \) is a satisfying approximation of \( \psi_1(0) \).
For some \( \varepsilon_1 \) small enough, the concentration ratio of \( \psi_1(\varepsilon_1) \) becomes close enough (down to some prescribed tolerance \( \eta \)) to the first eigenvalue \( \lambda_1(0) \) of \( \mathbf{K}(0) \).
In this case, we consider that \( \psi_1(\varepsilon_1) \) is a good enough approximation of \( \psi_1(0) \), and we record this vector \( \psi_1(\varepsilon_1) \). 
We then repeat this process in order to find an approximation of the second eigenvector \( \psi_2(0) \) of \( \mathbf{K}(0) \).
Note that, when looking for \( \psi_2(\varepsilon) \) an approximation of \( \psi_2(0) \), we can look for it in the orthogonal of \( \Span\{\psi_1(\varepsilon_1)\} \). 
This will guarantee that the set of vector we obtain at the end of the procedure is indeed a basis of \( \mathbb{C}^N \). 

The orthogonalization \( \tilde{w} \) of a vector \( w \) with respect to a vector \( v \) is simply given by \( \tilde{w} = \frac{u}{\sqrt{u^* u}} \) where \( u = (1 - \frac{vv^*}{v^* v}) w \).
In practice, we use an iterative process to approximate eigenvectors, so the orthogonalization is required at each iteration.

By construction of this basis of approximate eigenvectors, the concentration ratio of the \( n \)-th vector \( \psi_n(\varepsilon_n) \) is \( \eta \)-close to the true \( n \)-th eigenvalue \( \lambda_n(0) \).
Another advantage of this procedure is that we do not need to know beforehand all the exact eigenvalues, since it is possible to work one eigenvector at a time.

The procedure is described above in the one-dimensional case, but the same ideas apply in the multi-dimensional setting. 
Note that some care has to be taken when choosing the set-valued function \( \Omega_i(\cdot) \), whose role is to mimic the scaling of the interval when \( d=1 \). For example, if \( \Omega_i \) has ``holes'', then we don't want the holes to move, only possibly to get bigger or smaller.
This explains why, in the case \( \Omega_S = \text{cat-head} \) given in Section \ref{subsect:slepian -- numerical examples - cat-head}, the set-valued function \( \Omega_S(\cdot) \) cannot simply be a scaling of \( \Omega_S \). Otherwise, the holes within the cat-head shape would move, and the convergence would require smaller values of \( \varepsilon \).

The algorithm is given by Algorithm \ref{algo: slepian -- varying mask multi-D}.

\begin{remark}
    When looking for an eigenvector \( \psi(\varepsilon) \) of \( \mathbf{K}(\varepsilon) \), one can take advantage of the eigenvector obtained for a previous (larger) value of \( \varepsilon \), in order to start the search of an approximate eigenvector and thus speed up computations.
\end{remark}

\begin{remark}
    In Algorithm \ref{algo: slepian -- varying mask multi-D}, we allow considering approximate eigenvalues of the matrix \( \mathbf{K}(0) \). 
    This is due to the fact that, with high-dimensional matrices, it is often difficult to obtain the eigenvalues precisely because of the computational cost.
    Thus we allow approximate eigenvalues so that one can speed up the computations of \( \lambda_n(0) \) by only looking for an \( \eta \)-close approximation.
\end{remark}

\begin{algorithm}
    \caption{Varying masks method}
    \label{algo: slepian -- varying mask multi-D}
    \begin{algorithmic}
        \Require \\
        \begin{itemize}
            \item \( \Omega_S, \Omega_F \): two finite-volume domains of \( \mathbb{R}^d \).
            \item \( \varepsilon \mapsto \Omega_i(\varepsilon) \): a set-valued function, decreasing for the relation of set inclusion, such that \( \Omega_i(0) = \Omega_i \) and \( \Omega_i(+\infty) = Z_i \) with \( \text{meas}(Z_i)=0 \), \( i\in \{S, F\} \).
            \item \( m_S(\varepsilon, \cdot) = \mathbf{1}_{\Omega_S(\varepsilon)} \) and \( \widehat{m_{F}}(\varepsilon, \cdot) = \mathbf{1}_{\Omega_F(\varepsilon)} \): the modified masks.
            \item \( N \): the number of discretization points for each dimension.
            \item \( \varepsilon \mapsto \mathbf{K}(\varepsilon) \): the modified concentration matrix. It is a square matrix of size \( N^d \), corresponding to the modified masks \( m_S(\varepsilon, \cdot) \) and \( \widehat{m_F}(\varepsilon, \cdot) \).
            \item \( M \): number of eigenvectors we are looking for, \( M \leq N \).
            \item \( \varepsilon_{\max}, \varepsilon_{\min} \): maximum and minimum value of the parameter \( \varepsilon \).
            \item \( \left\{ \varepsilon_T, \dots, \varepsilon_1 \right\} \subset [\varepsilon_{\min}, \varepsilon_{\max}] \): \( T \) discretization points of the interval \( [\varepsilon_{\min}, \varepsilon_{\max}] \) (can be a uniform discretization, \( \log \)-uniform, \dots). They are assumed to be such that \( \varepsilon_t > \varepsilon_{t-1} \), \( t \in \llbracket 1, T \rrbracket \).
            \item \( \eta \): numerical tolerance to compare two eigenvalues.
        \end{itemize}
        \State{\( q := 0 \): this is the number of recorded eigenvectors yet.}
        \State{\( \alpha_{saved} = \left(\alpha_{saved, 1}, \dots, \alpha_{saved, M}\right) \): the vector to hold all the concentration ratios.}
        \State{\( v_{saved} = \left(v_{saved, 1}, \dots, v_{saved, M}\right) \): the matrix to hold all the recorded eigenvectors}
        \For{\( \varepsilon = \varepsilon_T, \dots, \varepsilon_0 \)}
        \State{Find \( \widetilde{\lambda_q}(0) \), an \( \eta \)-approximation of the exact eigenvalue \( \lambda_q(0) \). A standard eigenalgorithm can be used here.}
        \State{Find \( (\kappa, u) \) the most significant eigenpair of \( \mathbf{K}(\varepsilon) \) (i.e. the one associated to the eigenvalue of highest magnitude), where \( u \perp \Span\left\{ v_{saved, 1}, \dots, v_{saved, q} \right\} \) and \( |u|_2=1 \). A standard eigenalgorithm can be used here.}
        \State{Compute the concentration ratio with respect to the unperturbed problem:
            \begin{equation*}
                \nu := \frac{u^* \mathbf{K}(0) u}{u^* u}.
            \end{equation*}
        }
        \If{\( \left| \nu - \widetilde{\lambda_q}(0) \right| \leq \eta \)}
        \State{\( v_{saved, q} \leftarrow u \).}
        \State{\( \alpha_{saved, q} \leftarrow \nu \).}
        \State{\( q \leftarrow q+1 \).}
        \EndIf
        \State{Stop if \( q \geq M \).}
        \EndFor
    \end{algorithmic}
\end{algorithm}

Under our current assumptions, the fact that the procedure described above yields a correct approximation of the eigenvectors is shown by Lemma \ref{lemma: algo gives correct approximation}.

\begin{lemma}
    \label{lemma: algo gives correct approximation}
    Let \( \mathbf{A} \in \mathbb{R}^{N\times N} \) an Hermitian matrix, and denote \( \left\{ \lambda_i \right\}_{i=1}^N \) its eigenvalues (they are all real), ordered so that \( \lambda_i \geq \lambda_{i+1} \).
    Let \( \left\{ v_1, \dots, v_N \right\} \) an orthonormal basis of \( \mathbb{C}^N \), where \( v_i \in \mathbb{C}^N \) is an eigenvector of \( \mathbf{A} \) associated to \( \lambda_i \).
    Let \( \eta > 0 \) and \( w \in \mathbb{C}^N \), such that \( |w|_2 = \sqrt{w^* w} = 1 \) and
    \begin{equation}
        \label{eqn: lemma correct approximation -- criterion on concentration ratio}
        \left| w^* \mathbf{A} w - \lambda_1 \right| \leq \eta.
    \end{equation}
    For \( m\leq N \), we have
    \begin{equation*}
        \left| w - \mathrm{Proj}_{\Span\left\{ v_1, \dots, v_m \right\}} w \right|_2^2
        \leq \frac{\eta}{\lambda_1 - \lambda_{m+1}} .
    \end{equation*}
\end{lemma}

\begin{proof}
    Decompose \( w \) into the \( \{ v_i \}_{i=1}^N \) basis: \( w = \sum_{i=1}^N c_i v_i \), for some coefficients \( c_i\in \mathbb{C} \) such that \( \sum_{i=1}^N |c_i|^2 = 1 \).
    Owing to
    \begin{equation*}
        w^* \mathbf{A} w
        = \left( \sum_{i=1}^N c_i v_i \right)^* \left( \sum_{j=1}^N c_j \lambda_j v_j \right) = \sum_{i=1}^N |c_i|^2 \lambda_i,
    \end{equation*}
    we get
    \begin{align*}
        \lambda_1 - w^* \mathbf{A} w
        &= \sum_{i=1}^N |c_i|^2 \left( \lambda_1 - \lambda_i \right) = \sum_{i=1}^m |c_i|^2 \left( \lambda_1 - \lambda_i \right) + \sum_{i=m+1}^N |c_i|^2 \left( \lambda_1 - \lambda_i \right) \\
        &\geq \sum_{i=1}^m |c_i|^2 \left( \lambda_1 - \lambda_i \right) + 
        \sum_{i=m+1}^N |c_i|^2 \left( \lambda_1 - \lambda_{m+1} \right) \\
        &= \sum_{i=1}^m |c_i|^2 \left( \lambda_1 - \lambda_i \right) + \left( 1 - \sum_{i=1}^m |c_i|^2 \right) \left( \lambda_1 - \lambda_{m+1} \right) \\
        &= \sum_{i=1}^m |c_i|^2 \left( \lambda_{m+1} - \lambda_i \right) + \lambda_1 - \lambda_{m+1}.
    \end{align*}
    Use the ordering of eigenvalues to obtain
    \begin{align*}
        \eta \geq \lambda_1 - w^* \mathbf{A} w
        > (\lambda_1 - \lambda_{m+1}) \sum_{i=m+1}^N |c_i|^2
    \end{align*}

    Therefore,
    \begin{equation*}
        \left| w - \mathrm{Proj}_{\Span\left\{ v_1, \dots, v_m \right\}} w \right|_2^2
        = \left| \sum_{i=m+1}^N c_i v_i \right|_2^2 = \sum_{i=m+1}^N |c_i|^2 \leq \frac{\eta}{\lambda_1 - \lambda_{m+1}}.
    \end{equation*}
\end{proof}

\begin{corollary}
    Under the same assumptions as Lemma \ref{lemma: algo gives correct approximation}, we have
    \begin{equation*}
        \left| \mathbf{A}w - \lambda_1 w \right|_2 \leq \frac{\lambda_1 - \lambda_n}{\lambda_1 - \lambda_2} \eta.
    \end{equation*}
\end{corollary}

\begin{proof}
    By Lemma \ref{lemma: algo gives correct approximation}, we can write \( w = r + c_1 v_1 \), where \( c_i \in \mathbb{C} \) and \( r\in \mathbb{C}^d \) such that \( |r|_2 \leq \frac{\eta}{\lambda_1 - \lambda_{2}}  \).
    Thus,
    \begin{equation*}
        \mathbf{A} w = \mathbf{A}r + \lambda_1 c_1 v_1
    \end{equation*}
    and 
    \begin{equation*}
        \mathbf{A}w - \lambda_1 w = \mathbf{A}r - \lambda_1 r.
    \end{equation*}
    Finally,
    \begin{equation*}
        \left| \mathbf{A}w - \lambda w \right|_2 \leq |||\mathbf{A} - \lambda_1 I|||_2 |r|_2 \leq \frac{\lambda_1 - \lambda_n}{\lambda_1 - \lambda_2} \eta,
    \end{equation*}
    where we have used the estimate of \( |r|_2 \).
\end{proof}

\begin{remark}
    Lemma \ref{lemma: algo gives correct approximation} is useful numerically: if one considers the criterion \eqref{eqn: lemma correct approximation -- criterion on concentration ratio} to select the \( i \)-th eigenvector \( v_i \), then the numerical tolerance \( \eta \) has to be small compared the difference \( \lambda_i - \lambda_{i+1} \) in order to have a good approximation of \( v_i \).
    In particular, if \eqref{eqn: lemma correct approximation -- criterion on concentration ratio} is satisfied with \( \eta \ll \lambda_1 - \lambda_{2} \), then we have a good approximation of \( v_1 \).
\end{remark}

\begin{remark}[Continuation methods]
    Since we are considering a perturbed problem for \( \varepsilon>0 \), it seems natural to have in mind continuation methods. Here, the clusters of eigenvalues at zero and one force us to consider continuation methods where eigenvalues have to be distinguished using their derivatives because their value is not enough.
    This type of method has for instance been introduced and used in \cites{limEigenvectorDerivativesRepeated1989,juangEigenvalueEigenvectorDerivatives1989}.
    However, unreported numerical experiments showed that the approximate eigenvectors obtained are not significantly different from those presented here with the varying masks procedure.
    Moreover, continuation methods have some drawbacks that are naturally fixed using the varying masks procedure:
    \begin{itemize}
        \item we cannot control how close the concentration ratios are to the true eigenvalues;
        \item in order to compute the eigenvector derivative, all eigenvectors are needed, which is very costly in two- and higher-dimensional settings;
        \item continuation methods require a concentration matrix that is differentiable with respect to the perturbation parameter \( \varepsilon \), but the indicator masks considered are not differentiable. Thus, applying continuation methods would only help us solve an approximate concentration problem where the masks would be smooth (with respect to \( \varepsilon \));
        \item the orthogonalization of eigenvectors is not guaranteed with eigenvector continuation (though some work could probably be done in that regard to alleviate this).
    \end{itemize}
\end{remark}

\section{Numerical examples}
\label{sect:slepian -- numerical examples}

In this section we will compare the eigenvectors obtained from a standard eigendecomposition of the concentration matrix \( \mathbf{K}(0) \), with the approximate eigenvectors obtained via the varying masks procedure described in Algorithm \ref{algo: slepian -- varying mask multi-D}.

Note that the eigenvectors are \textit{a priori} complex vectors, but the Fourier restrictions we consider in the following numerical examples are chosen even so that the eigenvectors can actually be chosen real. 
This is why we only present the real part of the (approximate) eigenvectors and don't mention their imaginary part. 

\bigskip    

We present in this section the results obtained using Algorithm \ref{algo: slepian -- varying mask multi-D}, where both space and Fourier masks are varying, with the same variation \( \mu(\varepsilon) \). 
During our numerical experiments we have also tried to let only one of the two masks vary. 
However, the results were always worse than with the simultaneously varying masks.
We emphasize that Algorithm \ref{algo: slepian -- varying mask multi-D} is not rigorously justified, and we propose it only because it yielded interesting results that were better than a standard eigenalgorithm.
For instance, we don't know if the choice of \( \mu(\varepsilon) \) is the best, nor if it is better to have different variations \( \mu_S \) and \( \mu_F \) for the space and Fourier masks.

\subsection{Interval masks -- Moderate Fourier restriction}
\label{subsect:slepian -- numerical examples - interval 0.3}

This is a one-dimensional example, and we choose the number of discretization points \( N = 150 \). This means that the matrix is rather small, so we can have a very small numerical tolerance, hence we choose \( \eta = 10^{-10} \).

The \( \varepsilon \)-discretization is a log-uniform discretization of the interval \( [10^{-1}, 10^2] \), with \( T=250 \) discretization points.
The space mask is \( m_S(\varepsilon, \cdot) = \mathbf{1}_{[-\mu(\varepsilon), \mu(\varepsilon)]} \) where \( \mu(\varepsilon) \) is given by \eqref{eqn:slepian1d -- mu eps}, and the Fourier mask is \( \widehat{m_F}(\varepsilon, \cdot) = \mathbf{1}_{[-\omega\mu(\varepsilon), \omega\mu(\varepsilon)]} \) where \( \omega = 0.3 \cdot 2\pi \).

We plot in Figure \ref{fig: numerical example -- 1d interval mF 0.3 - 16 eigenvectors eigendecomposition} the eigenvectors obtained with a standard eigendecomposition of \( \mathbf{K}(0) \) (solid blue curve) as well as the exact eigenvectors obtained via \( \mathbf{P} \) (dashed orange curve). 
We can note that the eigendecomposition yields eigenvectors that are not localized, and they are also not symmetric (neither even nor odd).

The approximate eigenvectors obtained via the varying masks procedure are given in Figure \ref{fig: numerical example -- 1d interval mF 0.3 - 16 eigenvectors varying mask}.
They are much closer to the exact eigenvectors, and exhibit the expected localization and symmetry properties.
Note that we have not postprocessed the approximate eigenvectors, and in particular we have not applied any normalization convention. This explains why some approximate eigenvectors are the opposite of the desired ones.
Moreover, the concentration ratios of the approximate eigenvectors are guaranteed to be \( \eta \)-close to the true eigenvalues, so the varying masks procedure is a satisfying robust alternative to the eigendecomposition.

\vspace{\baselineskip}

\begin{figure}
    \centering
    \includegraphics[width=\linewidth]{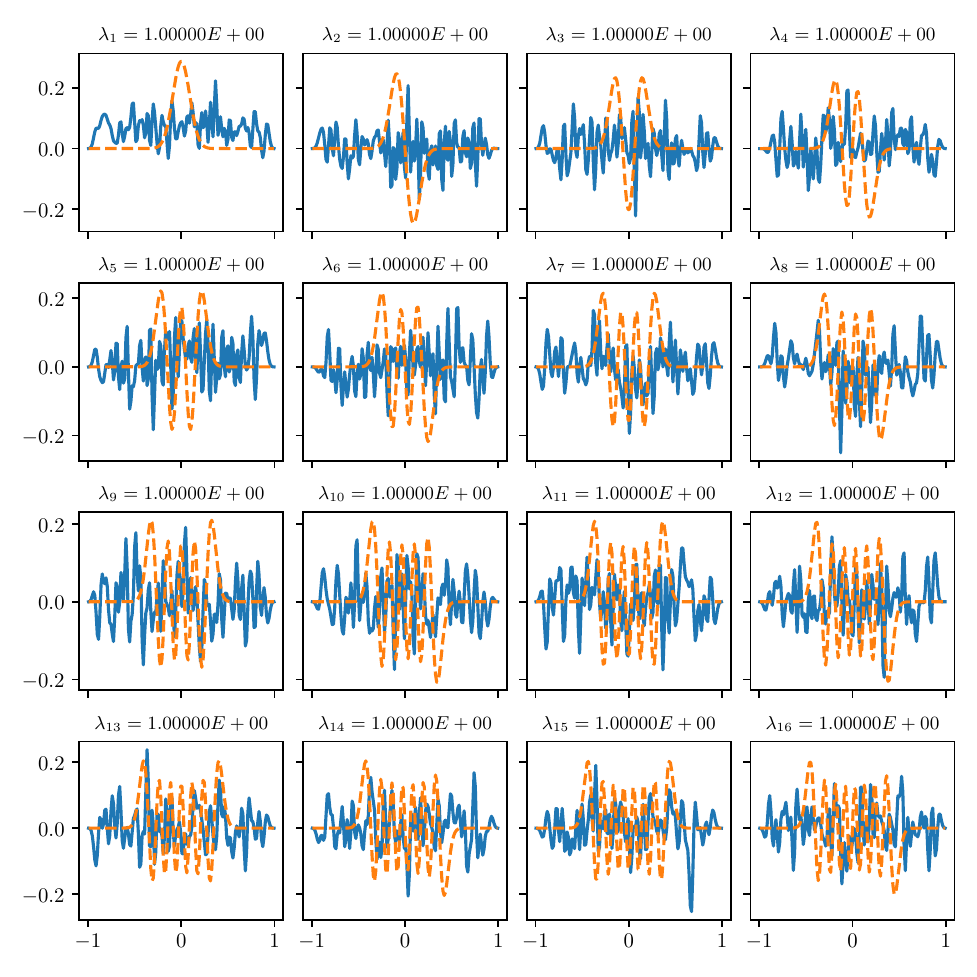}
    \caption{16 first eigenvectors obtained with a standard eigendecomposition (solid blue curve) of \( \mathbf{K}(0) \), compared to the exact eigenvectors (dashed orange curve). Here, \( \widehat{m_{F}} = \mathbf{1}_{[-0.3\cdot 2\pi\mu(\varepsilon), 0.3\cdot 2\pi \mu(\varepsilon)]} \). \( N_1 = 150, \eta = 10^{-10} \).}
    \label{fig: numerical example -- 1d interval mF 0.3 - 16 eigenvectors eigendecomposition}
\end{figure}

\begin{figure}
    \centering
    \includegraphics[width=\linewidth]{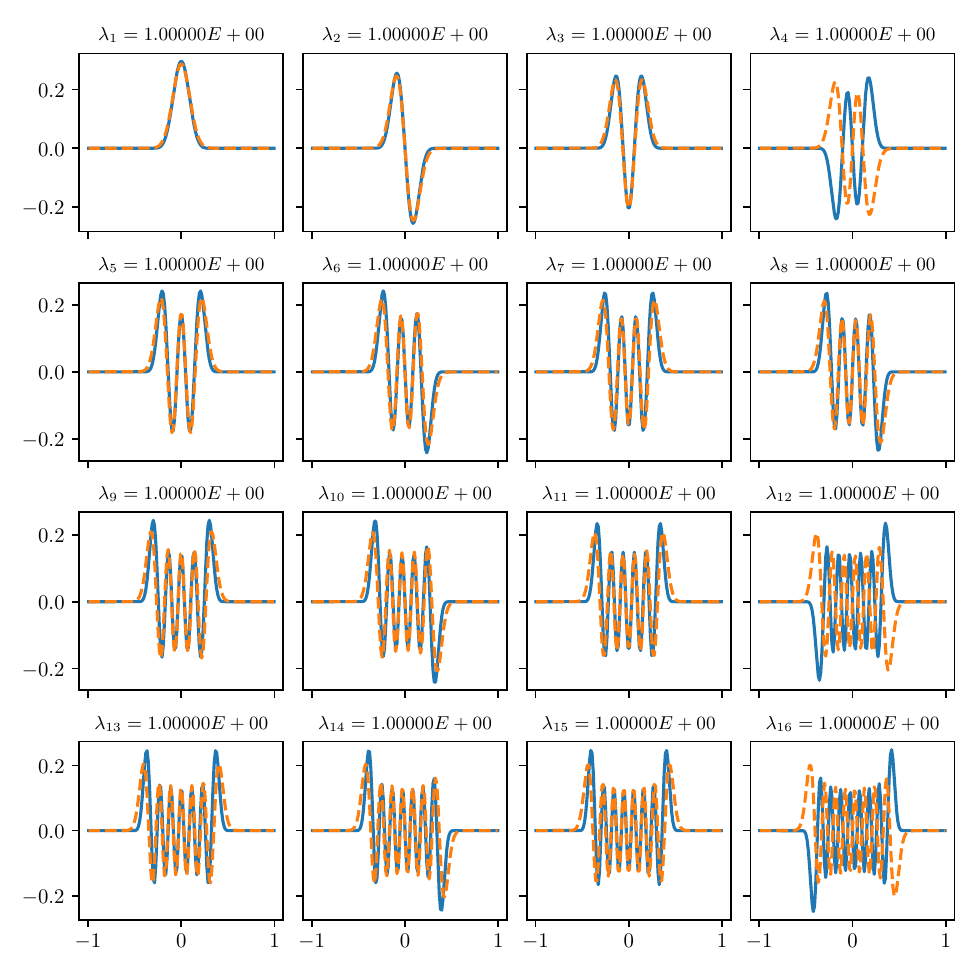}
    \caption{16 first eigenvectors obtained with the varying masks procedure (solid blue curve) of \( \mathbf{K}(0) \), compared to the exact eigenvectors (dashed orange curve). Here, \( \widehat{m_{F}} = \mathbf{1}_{[-0.3\cdot 2\pi\mu(\varepsilon), 0.3\cdot 2\pi\mu(\varepsilon)]} \). \( N_1 = 150, \eta = 10^{-10} \).}
    \label{fig: numerical example -- 1d interval mF 0.3 - 16 eigenvectors varying mask}
\end{figure}

\subsection{Interval masks -- Strong Fourier restriction}
\label{subsect:slepian -- numerical examples - interval 0.49}

For this second example, all the parameters of the concentration problem are the same as in Section \ref{subsect:slepian -- numerical examples - interval 0.3}, except for the Fourier restriction which is now \( \widehat{m_F} = \mathbf{1}_{[-\omega\mu(\varepsilon), \omega\mu(\varepsilon)]} \) with \( \omega = 0.49\cdot 2\pi \).
This is an important example, since almost all eigenvalues are clustered at one.

Figure \ref{fig: numerical example -- 1d interval mF 0.49 - 16 eigenvectors eigendecomposition} shows the eigendecomposition results, and Figure \ref{fig: numerical example -- 1d interval mF 0.49 - 16 eigenvectors varying mask} shows the results obtained with the varying masks procedure.
Once again, the varying masks procedure allows to recover approximate eigenvectors with a concentration ratio \( \eta \)-close to the exact eigenvalues of \( \mathbf{K}(0) \), and these vectors exhibit localization and symmetry properties that are lacking from the eigenvectors obtained via a standard eigendecomposition.
They are, however, slightly more localized than expected.

\begin{figure}
    \centering
    \includegraphics[width=\linewidth]{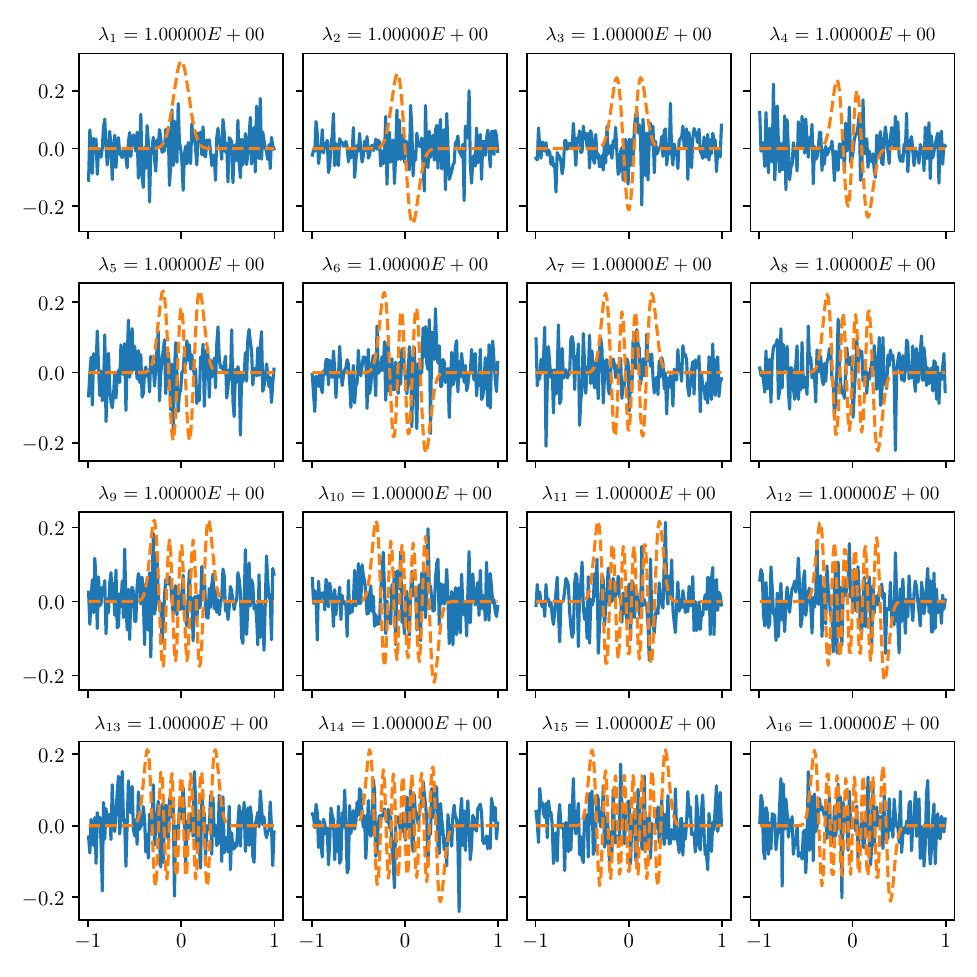}
    \caption{16 first eigenvectors obtained with a standard eigendecomposition (solid blue curve) of \( \mathbf{K}(0) \), compared to the exact eigenvectors (dashed orange curve). Here, \( \widehat{m_F} = \mathbf{1}_{[-0.49 \cdot 2\pi\mu(\varepsilon), 0.49\cdot 2\pi\mu(\varepsilon)]} \). \( N_1 = 150, \eta = 10^{-10} \).}
    \label{fig: numerical example -- 1d interval mF 0.49 - 16 eigenvectors eigendecomposition}
\end{figure}

\begin{figure}
    \centering
    \includegraphics[width=\linewidth]{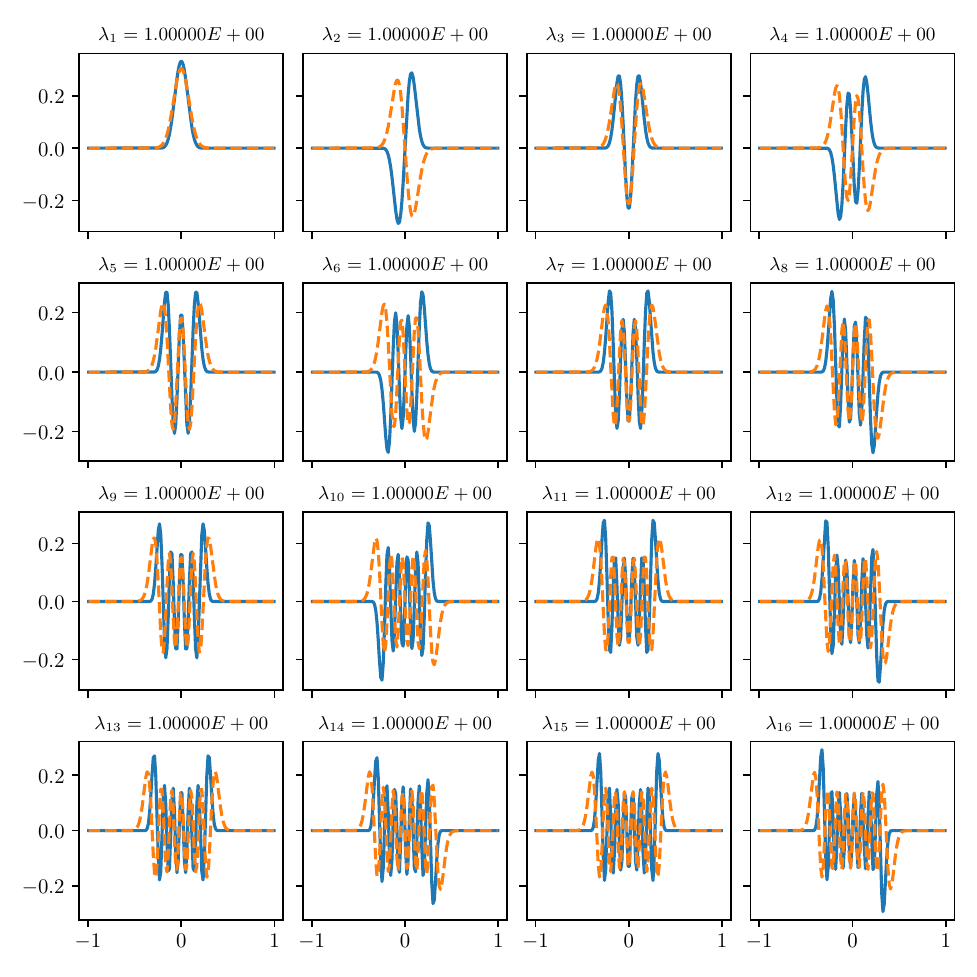}
    \caption{16 first eigenvectors obtained with the varying masks procedure (solid blue curve) of \( \mathbf{K}(0) \), compared to the exact eigenvectors (dashed orange curve). Here, \( \widehat{m_F} = \mathbf{1}_{[-0.49 \cdot 2\pi\mu(\varepsilon), 0.49\cdot 2\pi\mu(\varepsilon)]} \). \( N_1 = 150, \eta = 10^{-10} \).}
    \label{fig: numerical example -- 1d interval mF 0.49 - 16 eigenvectors varying mask}
\end{figure}

\subsection{Two-dimensional -- Centered balls}
\label{subsect:slepian -- numerical examples - centered balls}

We now consider the classical two-dimensional example of centered balls.
Since the matrix \( \mathbf{K} \) is a square matrix of size \( N^2 \), we have to keep this product reasonable for computational purposes. 
We recall that \( N \) is the number of discretization points in each dimension.
We choose \( N = 60 \), and \( \eta = 10^{-6} \).

The \( \varepsilon \)-discretization is a log-uniform discretization of the interval \( [10^{-1}, 10^1] \), with \( T=250 \) discretization points. 
We choose \( m_S(\varepsilon, \cdot) = \mathbf{1}_{B(0, 0.8 \mu(\varepsilon))} \), where \( \mu(\varepsilon) \) is again given by \eqref{eqn:slepian1d -- mu eps}, and \( \widehat{m_F}(\varepsilon, \cdot) = \mathbf{1}_{B(0, 0.3\cdot 2\pi \mu(\varepsilon))} \).

The results obtained using a standard eigendecomposition are given in Figure \ref{fig: numerical example --  2d centered ball mF 0.3 - 16 eigenvectors eigendecomposition}, and the Fourier transform of each eigenvector is given in Figure \ref{fig: numerical example --  2d centered ball mF 0.3 - 16 eigenvectors eigendecomposition FOURIER}.
We can note that the eigenvectors indeed have a compact support in space and that they are concentrated in the Fourier domain, but they do not exhibit the expected localization and symmetry properties.
For instance, the first eigenvector is not more localized than the sixteenth

The results obtained using the varying masks procedure are given in Figure \ref{fig: numerical example --  2d centered ball mF 0.3 - 16 eigenvectors varying mask}, and the Fourier transform of each vector is given in Figure \ref{fig: numerical example --  2d centered ball mF 0.3 - 16 eigenvectors varying mask FOURIER}.
These approximate eigenvectors do exhibit the expected localization and symmetry properties, illustrating once again that the varying masks procedure is a robust approximate alternative to the eigendecomposition of the initial concentration matrix.

\begin{figure}
    \centering
    \includegraphics[width=\linewidth]{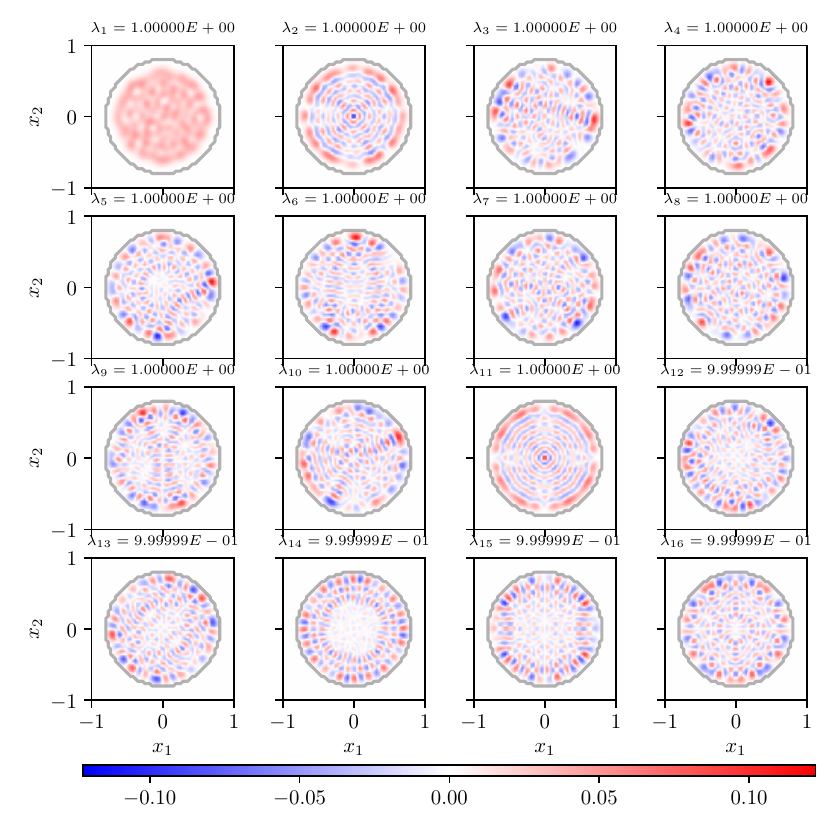}
    \caption{16 first eigenvectors obtained with a standard eigendecomposition of \( \mathbf{K}(0) \), where \( m_S(\varepsilon, \cdot) = \mathbf{1}_{B(0, 0.8 \mu(\varepsilon))} \) and \( \widehat{m_F}(\varepsilon, \cdot) = \mathbf{1}_{B(0, 0.3\cdot 2\pi \mu(\varepsilon))} \). 
    Here, \( m_S \) is outlined in gray. 
    \( N = 60, \eta = 10^{-6} \).}
    \label{fig: numerical example --  2d centered ball mF 0.3 - 16 eigenvectors eigendecomposition}
\end{figure}

\begin{figure}
    \centering
    \includegraphics[width=\linewidth]{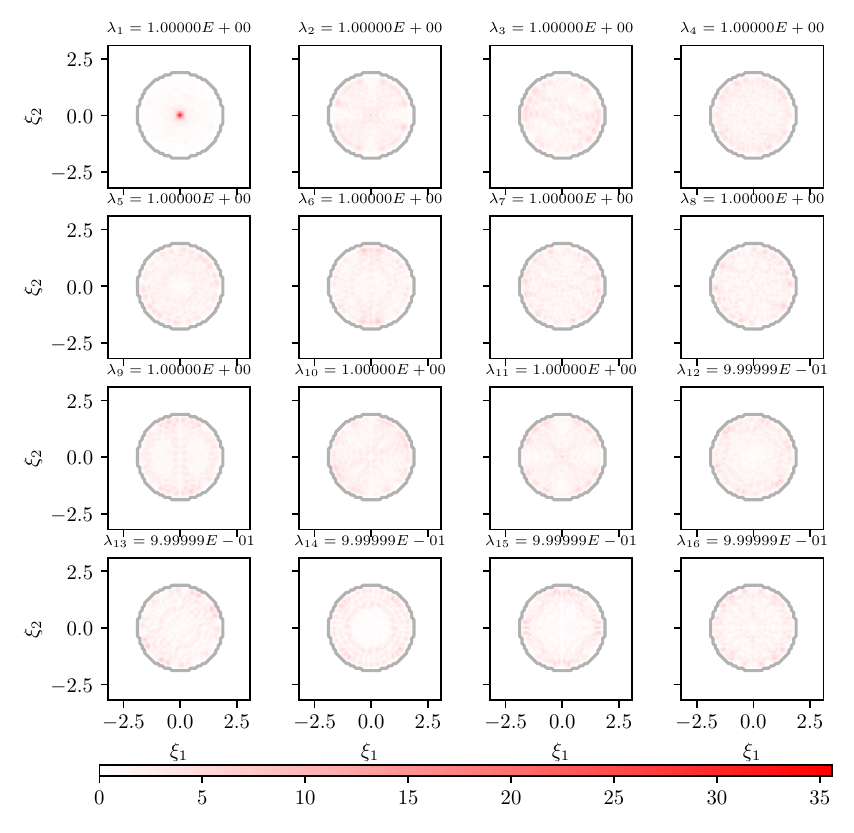}
    \caption{Absolute value of the Fourier transform of the 16 first eigenvectors obtained with a standard eigendecomposition of \( \mathbf{K}(0) \), where \( m_S(\varepsilon, \cdot) = \mathbf{1}_{B(0, 0.8 \mu(\varepsilon))} \) and \( \widehat{m_F}(\varepsilon, \cdot) = \mathbf{1}_{B(0, 0.3\cdot 2\pi \mu(\varepsilon))} \). 
    Here, \( \widehat{m_F} \) is outlined in gray. 
    \( N = 60, \eta = 10^{-6} \).}
    \label{fig: numerical example --  2d centered ball mF 0.3 - 16 eigenvectors eigendecomposition FOURIER}
\end{figure}

\begin{figure}
    \centering
    \includegraphics[width=\linewidth]{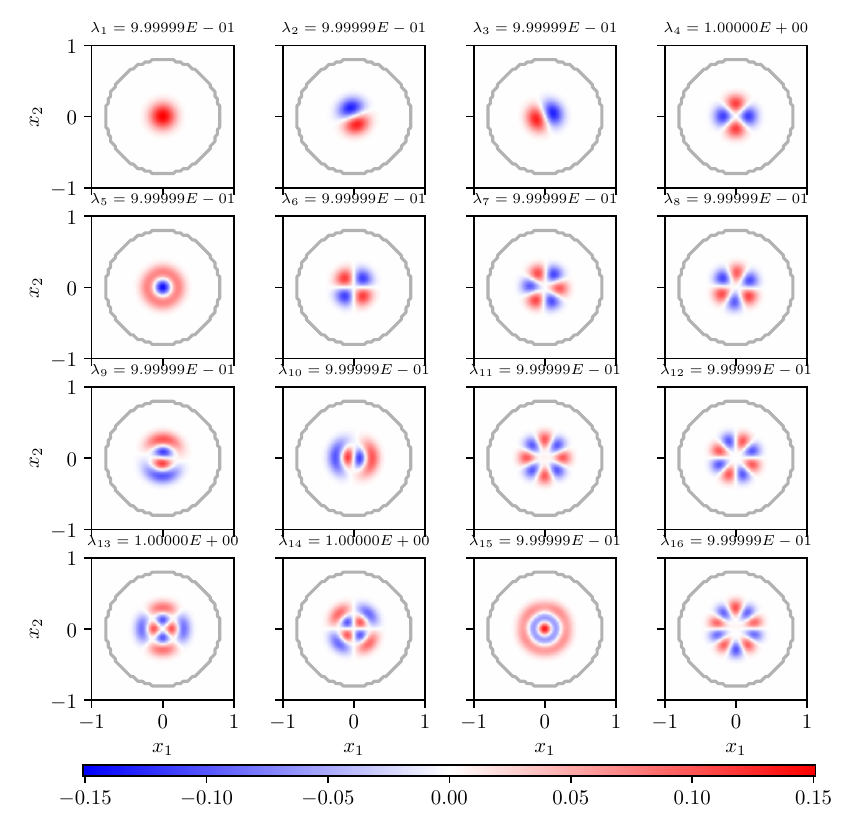}
    \caption{16 first eigenvectors obtained with the varying masks procedure of \( \mathbf{K}(0) \), where \( m_S(\varepsilon, \cdot) = \mathbf{1}_{B(0, 0.8 \mu(\varepsilon))} \) and \( \widehat{m_F}(\varepsilon, \cdot) = \mathbf{1}_{B(0, 0.3\cdot 2\pi \mu(\varepsilon))} \). 
    Here, \( m_S \) is outlined in gray. 
    \( N = 60, \eta = 10^{-6} \).}
    \label{fig: numerical example --  2d centered ball mF 0.3 - 16 eigenvectors varying mask}
\end{figure}

\begin{figure}
    \centering
    \includegraphics[width=\linewidth]{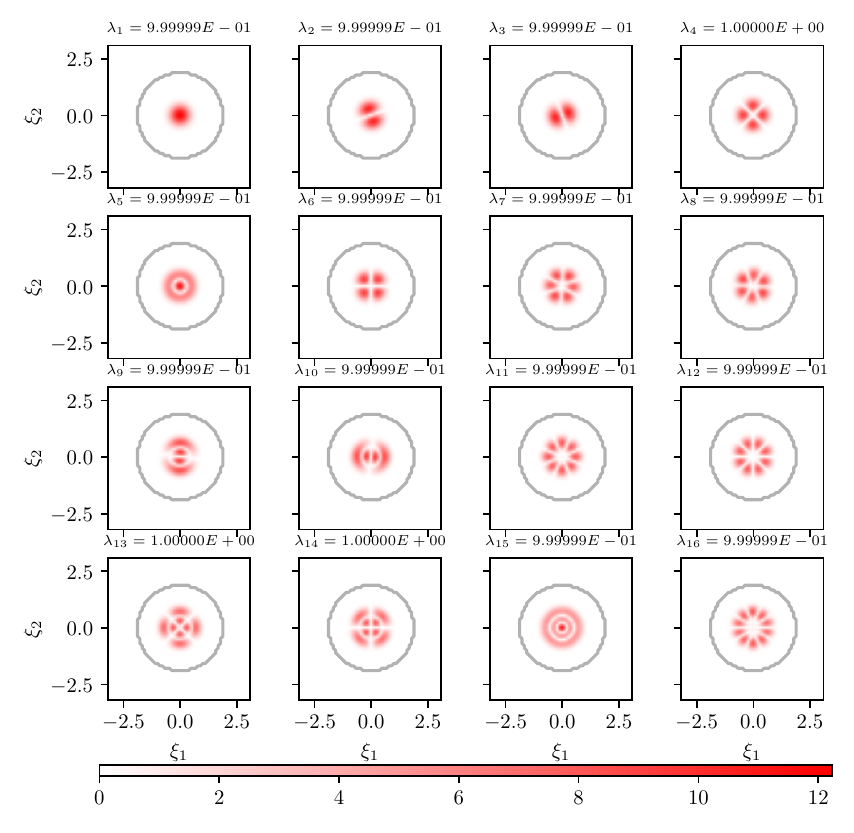}
    \caption{Absolute value of the Fourier transform of the 16 first eigenvectors obtained with the varying masks procedure of \( \mathbf{K}(0) \), where \( m_S(\varepsilon, \cdot) = \mathbf{1}_{B(0, 0.8 \mu(\varepsilon))} \) and \( \widehat{m_F}(\varepsilon, \cdot) = \mathbf{1}_{B(0, 0.3\cdot 2\pi \mu(\varepsilon))} \). 
    Here, \( \widehat{m_F} \) is outlined in gray. 
    \( N = 60, \eta = 10^{-6} \).}
    \label{fig: numerical example --  2d centered ball mF 0.3 - 16 eigenvectors varying mask FOURIER}
\end{figure}

\subsection{Two-dimensional -- cat-head shape}
\label{subsect:slepian -- numerical examples - cat-head}

This example is studied because of its sharp corners, for having a unique symmetry, and for having holes.
The set-valued function \( \Omega_S(\varepsilon) \) is shown in Figure \ref{fig: numerical example --  2d cat-head varying}. 
Note that it is not simply a scaling of \( \Omega_S = \Omega_S(0) \), because we have to take care of the holes in the domain. Indeed, we want these holes to not move, only to possibly grow or shrink.
This is to guarantee that an eigenvector \( v(\varepsilon) \) of \( \mathbf{K}(\varepsilon) \) with eigenvalue \( \lambda(\varepsilon) \) will have a concentration ratio \( \nu(v(\varepsilon)) \) with respect to \( \mathbf{K}(0) \) that is larger than \( \lambda(\varepsilon) \).

The results obtained via a standard eigendecomposition are given in Figure \ref{fig: numerical example --  2d cat-head mF 0.3 - 16 eigenvectors eigendecomposition}, and the Fourier transform of each eigenvector is given in Figure \ref{fig: numerical example --  2d cat-head mF 0.3 - 16 eigenvectors eigendecomposition FOURIER}.
The results obtained via the varying masks procedure are given in Figure \ref{fig: numerical example --  2d cat-head mF 0.3 - 16 eigenvectors varying mask}, and the Fourier transform of each approximate eigenvector is given in Figure \ref{fig: numerical example --  2d cat-head mF 0.3 - 16 eigenvectors varying mask FOURIER}. 

We can draw the same conclusions as those of Section \ref{subsect:slepian -- numerical examples - centered balls}: the eigendecomposition results in eigenvectors that are not localized as expected, and they also do not have the expected symmetry. The varying masks procedure yields approximate eigenvectors with a concentration ratio that is \( 2\eta \)-close to the exact eigenvalues (or \( \eta \)-close if we consider \( \lambda_q \) instead of \( \widetilde{\lambda_q} \) in Algorithm \ref{algo: slepian -- varying mask multi-D}), they exhibit the expected symmetry, and they also have the expected localization properties.

\begin{figure}
    \centering
    \includegraphics[width=0.5\linewidth]{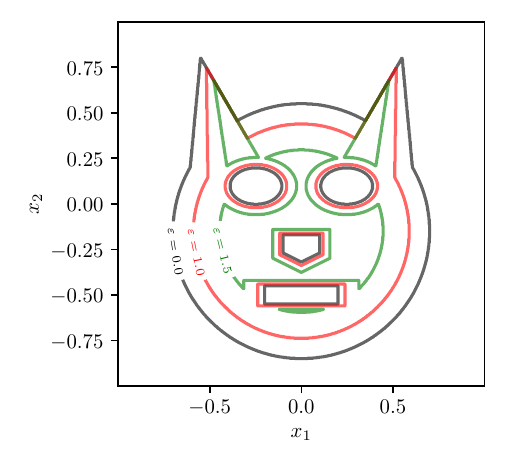}
    \caption{Set-valued function \( \Omega_S(\varepsilon) \) when \( \Omega_S = \Omega_S(0) = \text{cat-head} \), for \( \varepsilon\in \{0, 1, 3\slash 2 \} \).}
    \label{fig: numerical example --  2d cat-head varying}
\end{figure}

\begin{figure}
    \centering
    \includegraphics[width=\linewidth]{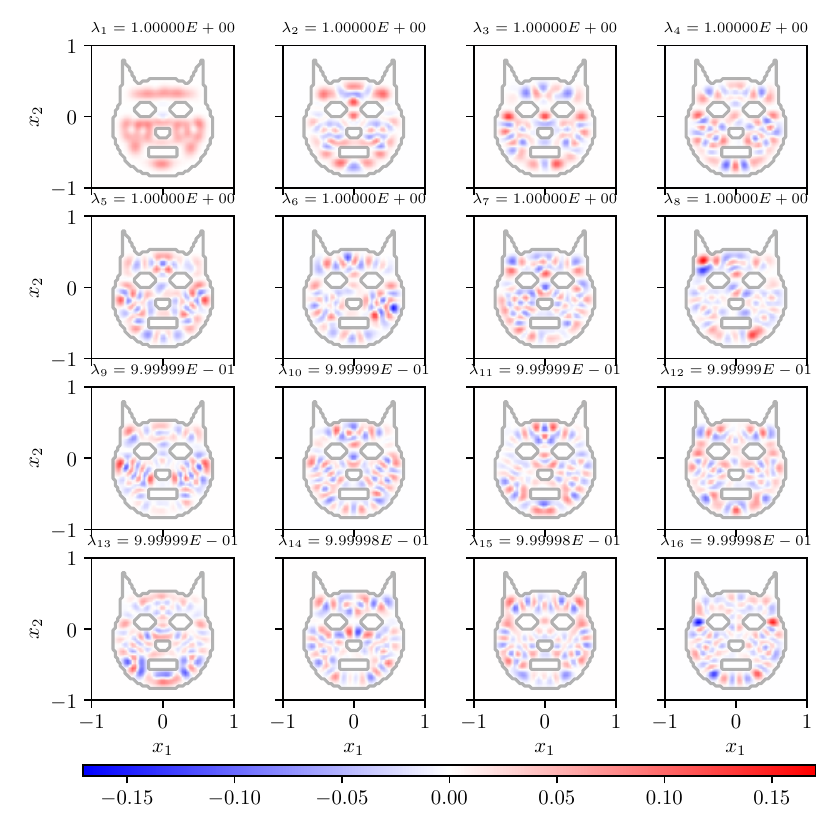}
    \caption{16 first eigenvectors obtained with a standard eigendecomposition of \( \mathbf{K}(0) \), where \( m_S(\varepsilon, \cdot) = \mathbf{1}_{\text{cat-head}(\varepsilon)} \) and \( \widehat{m_F}(\varepsilon, \cdot) = \mathbf{1}_{B(0, 0.3\cdot 2\pi \mu(\varepsilon))} \). 
    Here, \( m_S \) is outlined in gray. 
    \( N = 60, \eta = 10^{-6} \).}
    \label{fig: numerical example --  2d cat-head mF 0.3 - 16 eigenvectors eigendecomposition}
\end{figure}

\begin{figure}
    \centering
    \includegraphics[width=\linewidth]{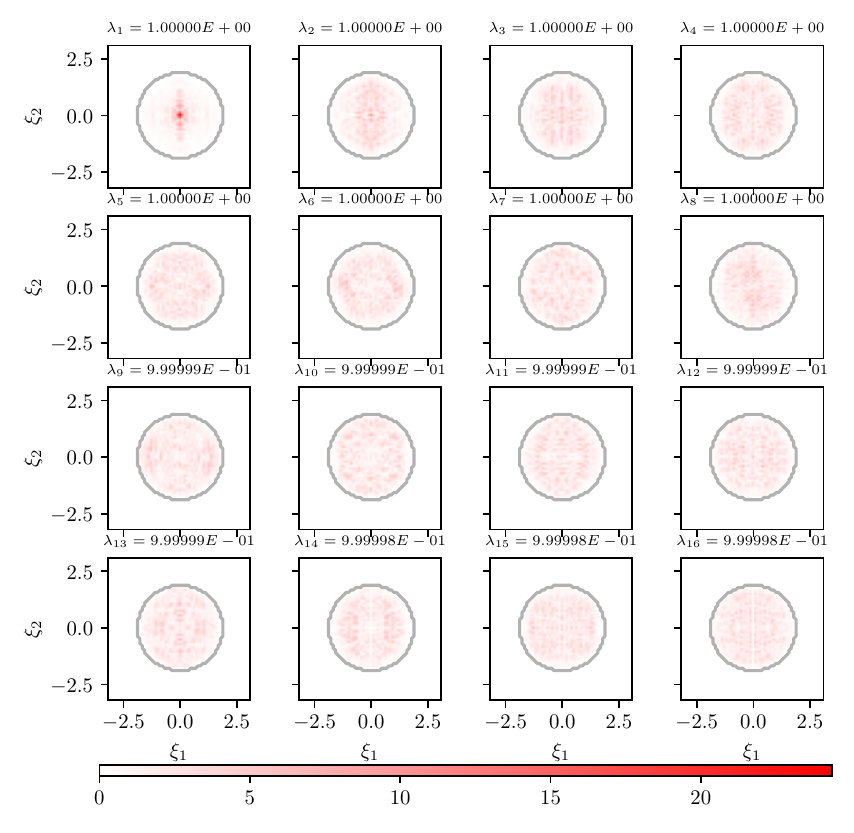}
    \caption{Absolute value of the Fourier transform of the 16 first eigenvectors obtained with a standard eigendecomposition of \( \mathbf{K}(0) \), where \( m_S(\varepsilon, \cdot) = \mathbf{1}_{\text{cat-head}(\varepsilon)} \) and \( \widehat{m_F}(\varepsilon, \cdot) = \mathbf{1}_{B(0, 0.3\cdot 2\pi \mu(\varepsilon))} \). 
    Here, \( \widehat{m_F} \) is outlined in gray.
    \( N = 60, \eta = 10^{-6} \).}
    \label{fig: numerical example --  2d cat-head mF 0.3 - 16 eigenvectors eigendecomposition FOURIER}
\end{figure}

\begin{figure}
    \centering
    \includegraphics[width=\linewidth]{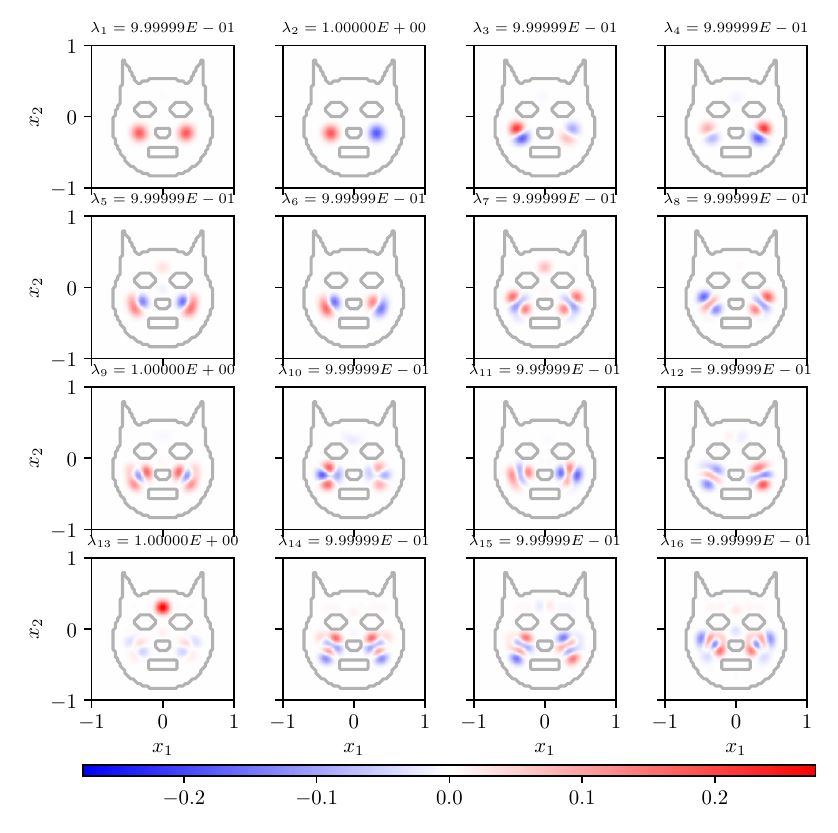}
    \caption{16 first eigenvectors obtained with the varying masks procedure of \( \mathbf{K}(0) \), where \( m_S(\varepsilon, \cdot) = \mathbf{1}_{\text{cat-head}(\varepsilon)} \) and \( \widehat{m_F}(\varepsilon, \cdot) = \mathbf{1}_{B(0, 0.3\cdot 2\pi \mu(\varepsilon))} \).
    Here, \( m_S \) is outlined in gray.
    \( N = 60, \eta = 10^{-6} \).}
    \label{fig: numerical example --  2d cat-head mF 0.3 - 16 eigenvectors varying mask}
\end{figure}

\begin{figure}
    \centering
    \includegraphics[width=\linewidth]{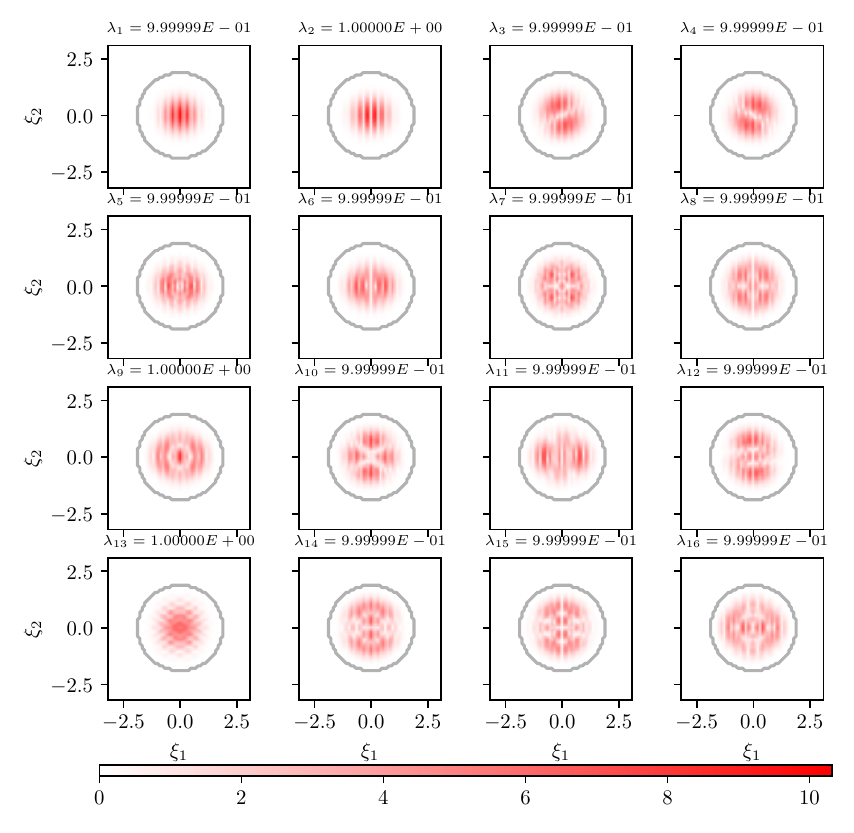}
    \caption{Absolute value of the Fourier transform of the 16 first eigenvectors obtained with the varying masks procedure of \( \mathbf{K}(0) \), where \( m_S(\varepsilon, \cdot) = \mathbf{1}_{\text{cat-head}(\varepsilon)} \) and \( \widehat{m_F}(\varepsilon, \cdot) = \mathbf{1}_{B(0, 0.3\cdot 2\pi \mu(\varepsilon))} \). 
    Here, \( \widehat{m_F} \) is outlined in gray. 
    \( N = 60, \eta = 10^{-6} \).}
    \label{fig: numerical example --  2d cat-head mF 0.3 - 16 eigenvectors varying mask FOURIER}
\end{figure}

\section{Perspectives and unanswered interrogations}
\label{sect:slepian -- conclusion}

In this work we have extended the framework laid by Slepian, Pollak, and Landau in the 1960s and 1970s, when they considered the problem of finding \( L^2 \) functions with the best simultaneous space and Fourier localization.
They solved the problem in a very elegant manner by looking at particular situations, but their ideas do not seem to be applicable in most cases.
We extended their work from balls in \( \mathbb{R}^d \) to general quadric domains in \( \mathbb{R}^d \), and considered a more general framework inspired by the natural {\em splitting representation} of the concentration operator \eqref{Ksplit} which allows to obtain exact formulas in the gaussian case and quasi-modes for masks close to the identity. 

We then introduced a new numerical algorithm for computing the eigenmodes of the concentration operator, named {\em varying masks method} which retains the idea of scaling progressively the masks from a situation where no relevant mode can be found (when both masks are too concentrated) to the targeted final situation. 
This simultaneous modification of the space and Fourier masks can be easily implemented in practice, and we reported excellent behavior of this algorithm. 
In particular, we showed on several numerical examples that the given procedure seems to be more robust than a standard eigendecomposition algorithm, and the results using this procedure exhibit better localization and symmetry properties than standard eigendecomposition algorithms.

However, the main idea of this work is based upon Assumption \ref{assumption: behavior of eigenvalues depending on varepsilon}, and some detailed analysis is required to determine if it actually holds.
Some analysis is also required to give the proposed algorithm some rigorous foundations, as well as study how different scalings in space and Fourier would affect the results.
Finally, we believe that the splitting approach \eqref{Ksplit} can be useful for more general understanding of the spectrum of the spectral concentration operator in very general situation.

\bibliography{biblio}

\end{document}